\RequirePackage{fix-cm}
\documentclass[smallcondensed, final]{svjour3}     
\smartqed  
\usepackage{graphicx}
\usepackage{lipsum}
\usepackage{commath}
\usepackage{amsfonts}
\usepackage{amssymb}
\usepackage{epstopdf}
\usepackage{algpseudocode}
\usepackage{algorithm}
\usepackage{convex}
\usepackage{xcolor}

\newtheorem{assumption}{Assumption}
\begin{document}

\title{On the sensitivity of the optimal partition for parametric second-order conic optimization\thanks{This work is supported by Air force Office of Scientific Research (AFOSR) Grant \# FA9550-15-1-0222.}
}

\titlerunning{On the sensitivity of the optimal partition for parametric SOCO}        

\author{Ali Mohammad-Nezhad         \and
        Tam\'as Terlaky 
}

\dedication{This paper is dedicated to Marco Lopez on the occasion of his $70th$ birthday.}


\institute{A. Mohammad-Nezhad \at
              Department of Mathematics, Purdue University, 150 N. University St., West Lafayette IN 47907  \\
              \email{mohamm42@purdue.edu}            \\
              Tel.: +(765)-496-3621\\
           \and
           T. Terlaky \at
              Department of Industrial and Systems Engineering, Lehigh University, 200 W. Packer Ave, Bethlehem, PA 18015 \\
               Tel.: +(610)-758-2903\\
              Fax: +1(610)758-4886\\
              \email{terlaky@lehigh.edu}  
}

\date{Received: date / Accepted: date}

\maketitle

\begin{abstract}
In this paper, using an optimal partition approach, we study the parametric analysis of a second-order conic optimization problem, where the objective function is perturbed along a fixed direction. We characterize the notions of so-called invariancy set and nonlinearity interval, which serve as stability regions of the optimal partition. We then propose, under the strict complementarity condition, an iterative procedure to compute a nonlinearity interval of the optimal partition. Furthermore, under primal and dual nondegeneracy conditions, we show that a boundary point of a nonlinearity interval can be numerically identified from a nonlinear reformulation of the parametric second-order conic optimization problem. Our theoretical results are supported by numerical experiments.
\end{abstract}

\section{Introduction}
In this paper, we investigate the \textit{optimal partition} approach for the parametric analysis of second-order conic optimization (SOCO) problems. Let $\mathcal{L}^{\bar{n}}_+\!:=\!\mathbb{L}^{n_1}_+ \times \mathbb{L}^{n_2}_+ \times \ldots \times \mathbb{L}^{n_p}_+$ be the Cartesian product of $p$ second-order cones~\cite{AG03}, where $\bar{n}:=\sum_{i=1}^p n_i$,
\begin{align*}
\mathbb{L}^{n_i}_+\!:=\!\big\{x^i\!:=\!(x^i_1,\ldots, x^i_{n_i})^T \in \mathbb{R}^{n_i} \mid x^i_1 \ge \|  x^i_{2:n_i}\| \big\}, 
\end{align*}
and $\|.\|$ is the $\ell_2$ norm. Parametric primal-dual SOCO problems are defined as 
\begin{align*}
&(\mathrm{P}_{\epsilon}) \qquad \inf_{x} \big\{(c + \epsilon \bar{c})^T x \mid Ax=b, \ x \in \mathcal{L}^{\bar{n}}_+ \big\}, \\
&(\mathrm{D}_{\epsilon}) \qquad \sup_{(y; s)} \big\{b^T y \mid  A^T y+s=c + \epsilon \bar{c}, \ s \in \mathcal{L}^{\bar{n}}_+\big\},
\end{align*}
in which $x\!:=\!(x^1; \ldots; x^p)$, $s\!:=\!(s^1;\ldots; s^p)$, $b \in \mathbb{R}^m$, $A\!:=\!(A^1, \ldots, A^p) \in \mathbb{R}^{m \times \bar{n}}$, where $A^i \in \mathbb{R}^{m \times n_i}$, $c\!:=\!(c^1; \ldots; c^p) \in \mathbb{R}^{\bar{n}}$ and $\bar{c}\!:=\!(\bar{c}^1; \ldots; \bar{c}^p) \in \mathbb{R}^{\bar{n}}$, where $c^i,\bar{c}^i \in \mathbb{R}^{n_i}$ for $i=1,\ldots,p$, and $(. \ ; \ldots ; .)$ denotes the concatenation of the column vectors. Note that $\epsilon \in \mathbb{R}$ is the perturbation parameter, and $\bar{c}$ is a fixed nonzero vector. Perturbation of the objective coefficient vector along a fixed direction is a basic sensitivity analysis question and is well-understood in the literature of linear optimization (LO), see e.g.,~\cite{RTV05}.

\vspace{5px}
\noindent
The optimal value of $(\mathrm{P}_{\epsilon})$ is denoted by $\psi(\epsilon)\!:\!\mathbb{R} \to \mathbb{R}\cup\{-\infty,\infty\}$, and it is called the \textit{optimal value function}%
\footnote{In this context, $\infty$ simply means that the primal problem is infeasible.}. Let $\mathcal{E} \subseteq \mathbb{R}$ be the domain of the optimal value function, i.e., the set of all $\epsilon$ such that $\psi(\epsilon) \in \mathbb{R}$. The interior point condition guarantees that $\mathcal{E}$ is nonempty and non-singleton.
\begin{assumption}\label{LID}
The coefficient matrix $A$ has full row rank.
\end{assumption}
\begin{assumption}\label{IPC}
 The interior point condition holds for both $(\mathrm{P}_{\epsilon})$ and $(\mathrm{D}_{\epsilon})$ at $\epsilon=0$, i.e., there exists a feasible solution $\big(x^{\circ}(0);y^{\circ}(0);s^{\circ}(0)\big)$ such that 
\begin{align*} 
x^{\circ i}(0),s^{\circ i}(0) \in \interior(\mathbb{L}^{n_i}_+), \qquad i=1,\ldots,p,
\end{align*}
where $\interior(\mathbb{L}^{n_i}_+)\!:=\!\Big\{x^i \in \mathbb{R}^{n_i} \mid x^i_1 > \|  x^i_{2:n_i}\|\Big\}$ denotes the interior of $\mathbb{L}^{n_i}_+$.
 \end{assumption}

\vspace{5px}
\noindent
By~\cite[Lemma~3.1]{GS99} and a theorem of the alternative~\cite[Lemma~12.6]{CSW13}, Assumptions~\ref{LID} and~\ref{IPC} imply the existence of an interior solution $\big(x^{\circ}(\epsilon);y^{\circ}(\epsilon);s^{\circ}(\epsilon)\big)$ at every $\epsilon \in \interior(\mathcal{E})$. Furthermore, under Assumptions~\ref{LID} and~\ref{IPC}, $\psi(\epsilon)$ is proper and concave on $\mathcal{E}$~\cite[Lemma~2.2]{BJRT96} and continuous on $\interior(\mathcal{E})$~\cite[Corollary~2.109]{BS00}. Thus $\mathcal{E}$ is a closed, possibly unbounded, interval~\cite[Lemma~2.2]{BJRT96}. 

\vspace{5px}
\noindent
Given a fixed $\epsilon$, it is well-known that primal-dual interior point methods (IPMs)~\cite{NN94} can approximately solve $(\mathrm{P}_{\epsilon})$ and $(\mathrm{D}_{\epsilon})$ in polynomial time. We refer the reader to~\cite{AG03} and the references cited therein for a review of IPMs for SOCO, and to~\cite{ART03} for implementation of IPMs for SOCO. 

\subsection{Related works}
Sensitivity and stability analysis has been extensively studied for nonlinear optimization problems. Classical results about semicontinuity of the optimal set and the optimal value function date back to 1960's using the set-valued mapping theory~\cite{B97,Ho73b}. Zlobec et al.~\cite{ZGB82} identified the region of stability for perturbed convex optimization problems. The sensitivity of KKT solutions was studied by Fiacco~\cite{Fi76} and Fiacco and McCormick~\cite{FM90} under linear independence constraint qualification, second-order sufficient condition, and strict complementarity condition. Robinson~\cite{R82} released the strict complementarity condition but imposed a stronger second-order sufficient condition. Sensitivity analysis of nonlinear semidefinite optimization (SDO) and nonlinear SOCO problems has been widely studied in the past twenty years~\cite{BR2005,BS00,HMS20,MOS14,S97}. Under Slater and nondegeneracy conditions, by applying the \textit{implicit function theorem}~\cite{D60}, Shapiro~\cite{S97} established the differentiability of the unique optimal solution for a nonlinear convex SDO problem. Bonnans and Ram\'irez~\cite{BR2005} characterized strongly regular KKT solutions for nonlinear SOCO problems. Stability of local optimal solutions and constraint systems for nonlinear SOCO problems has been extensively studied in~\cite{HMS20,MOS14} through a variational analysis approach. We refer the reader to~\cite{BS00} for a comprehensive treatment of nonlinear conic optimization problems and to~\cite{Fi83} for a survey of classical results.

\vspace{5px}
\noindent
Sensitivity analysis based on the optimal partition%
\footnote{Inspired by the Goldman-Tucker theorem~\cite{GT56}, the concept of the optimal partition for LO is well-known in the literature of IPMs, see e.g.,~\cite{RTV05}.} is well-understood for LO~\cite{AM92,G94,JRT92} and linearly constrained quadratic optimization (LCQO) problems~\cite{BJRT96}. Compared to the optimal basis approach of LO which relies on nondegeneracy of the problem~\cite{JRT92,RTV05}, the optimal partition approach provides unique sensitivity analysis information regardless of any regularity condition. The optimal partition approach was initially studied by Adler and Monteiro~\cite{AM92}, Jansen et al.~\cite{JRT92}, and Greenberg~\cite{G94} for LO. The approach was extended to LCQO by Berkelaar et al.~\cite{BJRT96}. For LO and LCQO, $\interior(\mathcal{E})$ is partitioned into so-called invariancy sets. Generally speaking, an invariancy set is a maximal subset of $\interior(\mathcal{E})$, either a singleton or an open subinterval, on which the optimal partition is constant w.r.t. $\epsilon$. For a parametric SDO problem, Goldfarb and Scheinberg~\cite{GS99} investigated the differentiability of the optimal value function and provided auxiliary problems to compute the boundary points of an invariancy set. Yildirim~\cite{Yil04} extended the approach in~\cite{GS99} to linear conic optimization problems. Recently, Mohammad-Nezhad and Terlaky~\cite{MT20} introduced the concepts of a nonlinearity interval and transition point for the optimal partition of parametric SDO problems. Subsequently, Hauenstein et al.~\cite{HMTT19} proposed a numerical procedure, based on numerical algebraic geometry and ordinary differential equations, to partition $\interior(\mathcal{E})$ into the finite union of invariancy intervals, nonlinearity intervals, and transition points.

\subsection{Contribution}
While the notion of an optimal basis no longer exists for SOCO, the optimal partition is uniquely defined for any instance of SOCO which satisfies strong duality, regardless of strict complementarity and nondegeneracy conditions~\cite{BR2005}. Interestingly enough, the optimal partition of SOCO can be identified from a trajectory of interior solutions generated by a primal-dual IPM~\cite{TW14}. The optimal partition of a SOCO problem can be exploited to recover quadratic convergence to the unique optimal solution~\cite{MT19a}. Nevertheless, to date, only a few studies have been devoted to parametric analysis of linear conic optimization problems. In particular, the optimal partition and parametric analysis of $(\mathrm{P}_{\epsilon})-(\mathrm{D}_{\epsilon})$ have not been fully investigated in the literature. Unlike a univariate parametric LO problem~\cite{RTV05}, the optimal value function $\psi(\epsilon)$ is piecewise nonlinear, see e.g., problem~\eqref{ex:introductory_SOCO}, and to the best of our knowledge, there is no method in the literature for the computation of those nonlinear pieces.

\vspace{5px}
\noindent
Motivated by the study of parametric SDO in~\cite{GS99,HMTT19,MT20} and the identification of the optimal partition in~\cite{TW14}, we investigate the optimal partition approach for $(\mathrm{P}_{\epsilon})-(\mathrm{D}_{\epsilon})$ using its own algebraic structure, see Section~\ref{optimal_partition_definition}. Although $(\mathrm{P}_{\epsilon})$ and $(\mathrm{D}_{\epsilon})$ can be cast into a SDO formulation and studied independently using the results in~\cite{HMTT19,MT20}, see Section~\ref{conclusion}, we establish stronger results, in both theoretical and practical respects, when the optimal partition of $(\mathrm{P}_{\epsilon})-(\mathrm{D}_{\epsilon})$ is directly investigated. In doing so, we decompose $\interior(\mathcal{E})$ into the union of invariancy intervals, nonlinearity intervals, and transition points, to be defined in Section~\ref{sec:sensitivity_optimal_partition}. In our sensitivity analysis approach, invariancy intervals and transition points are natural extensions from parametric LO, whereas a nonlinearity interval is a newly defined concept. Invariancy and nonlinearity intervals are maximal subintervals of $\interior(\mathcal{E})$ on which the optimal partition is preserved under the given perturbation, while transition points are the boundary points of invariancy and nonlinearity intervals which belong to $\interior(\mathcal{E})$. 

\vspace{5px}
\noindent
Roughly speaking, our main contributions are:
 \begin{itemize}
 \item Characterization of nonlinearity intervals and transition points; 
 \item A numerical procedure for the computation of a nonlinearity interval; 
 \item Sufficient conditions for the identification of a transition point. 
\end{itemize}
 
\vspace{5px}
\noindent 
Unlike classical sensitivity analysis results for SOCO problems~\cite{BR2005,BS00} which only explore a small neighborhood of an optimal solution, our optimal partition approach (Algorithm~\ref{alg:nonlinearity_computation} and Theorem~\ref{transition_point_identification}) aims  to describe the optimal partition on the entire $\interior(\mathcal{E})$. More specifically, using the semi-algebraicity of the optimal set, we prove that the set of transition points is finite, see Theorem~\ref{transition_point_finiteness}. Furthermore, we provide sufficient conditions for the existence of a nonlinearity interval on the basis of Painlev\'e-Kuratowski set convergence, see Lemma~\ref{partition_inclusion_lemma} and Theorem~\ref{nonlinearity_continuity_SOCO}, and we show that continuity might fail on a nonlinearity interval, see problem~\eqref{SOCO_counterexample}. Under the existence of a strictly complementary optimal solution at a given $\bar{\epsilon}$, we solve nonlinear auxiliary problems alongside the so-called \textit{quantifier elimination} algorithm~\cite[Algorithm~14.5]{BPR06} to find a subinterval of the nonlinearity interval surrounding $\bar{\epsilon}$, see Algorithm~\ref{alg:nonlinearity_computation}. Finally, under primal and dual nondegeneracy conditions, we show that the derivative information from a nonlinear reformulation of $(\mathrm{D}_{\epsilon})$ can be invoked to identify a transition point of the optimal partition, see Theorem~\ref{transition_point_identification}.

\vspace{5px}
\noindent
Invariancy/nonlinearity intervals and transition points have important theoretical and practical implications. For instance, invariancy and nonlinearity intervals are the maximal subintervals of $\interior(\mathcal{E})$ on which $\psi(\epsilon)$ is linear and nonlinear, respectively, and transition points are the points at which at least either of strict complementarity or continuity of the optimal set mapping fails. Therefore, an IPM applied to $(\mathrm{P}_{\epsilon})-(\mathrm{D}_{\epsilon})$ at a transition point has potentially a weaker convergence behavior, because the Jacobian, see~\eqref{Jacobian_SOCO}, is always singular at a transition point. On the other hand, due to the lack of efficient warm-start procedures for IPMs, the identification of an invariancy/nonlinearity interval is important for “post-optimal analysis” and reoptimization of SOCO problems. In those applications, the optimal partition on the entire $\interior(\mathcal{E})$ can be characterized by utilizing an IPM only at a finite number of sample points in the interior of invariancy and nonlinearity intervals. This is in direct contrast with invoking an IPM for arbitrary points in $\interior(\mathcal{E})$, which is not only computationally expensive but also susceptible to missing a transition point.

\setlength{\abovedisplayskip}{2pt}
\setlength{\belowdisplayskip}{2pt}
\subsection{Organization of the paper}
The rest of this paper is organized as follows. Preliminaries are provided in Section~\ref{background}. In Sections~\ref{optimal_partition} through~\ref{nondegeneracy}, we provide background information about the complementarity, optimal partition, and nondegeneracy conditions; In Section~\ref{set_valued_analysis}, we review set-valued analysis and the continuity of the feasible set and optimal set mappings for $(\mathrm{P}_{\epsilon})$ and $(\mathrm{D}_{\epsilon})$. In Section~\ref{sec:sensitivity_optimal_partition}, we investigate the sensitivity of the optimal partition and optimal solutions w.r.t. $\epsilon$. In Section~\ref{sec:nonlinearity_transition}, we formally define the concepts of nonlinearity interval and transition point of the optimal partition; In Section~\ref{sec:continuity_nonlinearity}, we provide sufficient conditions for the existence of a nonlinearity interval. Furthermore, under strict complementarity condition, we present a numerical procedure for the computation of a nonlinearity interval; In Section~\ref{nonlinearity_interval_strict_fails}, under primal and dual nondegeneracy conditions, we show how to identify a transition point using higher-order derivatives of the Lagrange multipliers from a nonlinear reformulation of $(\mathrm{D}_{\epsilon})$. In Section~\ref{numerical_results}, we provide numerical results to demonstrate the convergence of lower and upper bounds generated by the numerical procedure and the magnitude of the derivatives. Our concluding remarks and directions for future research are summarized in Section~\ref{conclusion}. 
     
\paragraph{Notation}
We adopt the notation in accordance with~\cite{MT19a}, where $\mathbb{R}_+x$ is defined as
\begin{align*}
\mathbb{R}_+x\!:=\!\big\{\check{x} \mid \check{x}=\zeta x, \ \zeta \in \mathbb{R}_+\big\},
\end{align*}
and $R^i$ is the $n_i \times n_i$ diagonal matrix given by
\begin{align}
R^i\!:=\!\mathrm{diag}( 1, \ -1, \ldots, -1). \label{R_definition}
\end{align}
\noindent
For a convex set $\mathcal{C}$, the interior, relative interior, and the boundary are denoted by $\interior(\mathcal{C})$, $\ri(\mathcal{C})$, and $\boundary(\mathcal{C})$, respectively. Furthermore, $\mathbb{S}^{n_i}$ represents the vector space of $n_i \times n_i$ symmetric matrices, and $\mathbb{S}_+^{n_i}$ denotes the cone of $n_i \times n_i$ positive semidefinite matrices. Finally, letting $N^i$ be an $m \times n_i$ matrix for $i \in I \subseteq \{1,\ldots,p\}$, $(N^i)_I$ is an $m \times \sum_{i \in I}n_i$ matrix formed by the matrices $N^i$, for $i \in I$, put side by side.

\section{Preliminaries}\label{background}
\subsection{Optimality and complementarity}\label{optimal_partition}
The primal and dual optimal set mappings are defined as
\begin{align*}
\mathcal{P}^*(\epsilon)&\!:=\!\{x \mid (c + \epsilon \bar{c})^T x = \psi(\epsilon), \ Ax=b, \ x \in \mathcal{L}^{\bar{n}}_+\}, \\[-1\jot]
\mathcal{D}^*(\epsilon)&\!:=\!\{(y; s) \mid b^T y = \psi(\epsilon), \ A^T y+s=c + \epsilon \bar{c}, \ s \in \mathcal{L}^{\bar{n}}_+\}.
\end{align*}
Assumptions~\ref{LID} and~\ref{IPC} ensure that at every $\epsilon \in \interior(\mathcal{E})$ strong duality holds%
\footnote{In this paper, strong duality means that the duality gap is zero at optimality, and the optimal sets $\mathcal{P}^*(\epsilon)$ and $\mathcal{D}^*(\epsilon)$ are nonempty.}, and that both $\mathcal{P}^*(\epsilon)$ and $\mathcal{D}^*(\epsilon)$ are nonempty and compact, see e.g.,~\cite[Corollary~4.2]{T01}. Throughout this paper, $\big(x(\epsilon);y(\epsilon);s(\epsilon)\big) \in \mathcal{P}^*(\epsilon) \times \mathcal{D}^*(\epsilon)$ denotes a primal-dual optimal solution of $(\mathrm{P}_{\epsilon})$ and $(\mathrm{D}_{\epsilon})$. 

\vspace{5px}
\noindent
Under the strong duality assumption, $\mathcal{P}^*(\epsilon) \times \mathcal{D}^*(\epsilon)$ is the set of solutions of
\begin{align}\label{KKT_SOCO}
F\big((x;y;s)\big)\!:=\!\begin{pmatrix} Ax-b\\
A^T y+s-c - \epsilon \bar{c}\\
x \circ s
\end{pmatrix} = 0, \quad
x, s \in \mathcal{L}^{\bar{n}}_+,
\end{align}
where $\circ:\mathbb{R}^{\bar{n}} \times \mathbb{R}^{\bar{n}} \to \mathbb{R}^{\bar{n}}$ is a bilinear map~\cite{AG03} defined as
\begin{equation}\label{Jordan_prod}
x^i \circ s^i = L(x^i)s^i=L(s^i)x^i, \quad i=1,\ldots,p,
\end{equation}
in which
\begin{align*}
L(x^i)\!:=\!\begin{pmatrix} x_1^i & (x^i_{2:n_i})^T\\x^i_{2:n_i} & x_1^i I_{n_i-1} \end{pmatrix}
\end{align*}
is an \textit{arrow-head} symmetric matrix, $I_{n_i-1}$ is the identity matrix of size $n_i-1$, and $x \circ s\!:=\!(x^1 \circ s^1; \ldots; x^p \circ s^p)=0$ denotes the complementarity condition. It is immediate from the eigenstructure of $L(x^{i})$, see e.g.,~\cite[Theorem~3]{AG03}, that
\begin{align}\label{rank_structure}
\rank\!\big(L(x^{i})\big)=\begin{cases} n_i &x^{i} \in \interior(\mathbb{L}^{n_i}_+),\\ n_i-1 &x^{i} \in \boundary(\mathbb{L}^{n_i}_+)\setminus\{0\} ,\\ 0 &x^{i} =0, \end{cases} \qquad i=1,\ldots,p,
\end{align}
where $\boundary(\mathbb{L}^{n_i}_+)\!:=\!\big\{x^i \in \mathbb{L}^{n_i}_+ \mid x^i_1= \|  x^i_{2:n_i}\|\big\}$.

\vspace{5px}
\noindent
The Jacobian of equation system~\eqref{KKT_SOCO} is given by 
\begin{align}
\nabla F\big((x;y;s)\big)\!:=\!\begin{pmatrix} A & 0 & 0\\0 & A^T & I \\ L(s) & 0 & L(x) \end{pmatrix}, \label{Jacobian_SOCO}
\end{align}
where
\begin{equation}\label{blk_diag_arrow_hat}
\begin{aligned}
L\big(x\big)&\!:=\!\mathrm{diag}\big(L(x^{1}), \ldots, L(x^{p})\big),\\
L\big(s\big)&\!:=\!\mathrm{diag}\big(L(s^{1}), \ldots, L(s^{p})\big).
\end{aligned}
\end{equation}

\vspace{5px}
\noindent
Among all primal-dual optimal solutions, we are interested in maximally and strictly complementary optimal solutions, which are indicated by superscript $^*$.
\begin{definition}\label{strict}
Let a primal-dual optimal solution $\big(x^*(\epsilon);y^*(\epsilon);s^*(\epsilon)\big) \in \mathcal{P}^*(\epsilon) \times \mathcal{D}^*(\epsilon)$ be given for a fixed $\epsilon$. Then $\big(x^*(\epsilon);y^*(\epsilon);s^*(\epsilon)\big)$ is called \textit{maximally complementary} if 
\begin{align*}
x^*(\epsilon) \in \ri\!\big(\mathcal{P}^*(\epsilon)\big) \quad \text{and} \quad \big(y^*(\epsilon);s^*(\epsilon)\big) \in \ri\!\big(\mathcal{D}^*(\epsilon)\big),
\end{align*}
and it is called \textit{strictly complementary} if $x^*(\epsilon)+s^*(\epsilon) \in \interior(\mathcal{L}^{\bar{n}}_+)$.
\end{definition}

\vspace{5px}
\noindent
A strictly complementary optimal solution may fail to exist at some $\epsilon$, see e.g., problem~\eqref{ex:introductory_SOCO}. However, under Assumptions~\ref{LID} and~\ref{IPC}, a maximally complementary optimal solution always exists for every $\epsilon \in \interior(\mathcal{E})$.

\begin{remark}
Throughout this paper, the strict complementarity condition is said to hold at $\epsilon$ if there exists a strictly complementary optimal solution $\big(x^*(\epsilon);y^*(\epsilon);s^*(\epsilon)\big)$.
\end{remark}

\subsection{Optimal partition}\label{optimal_partition_definition}
The notion of the optimal partition was originally defined for LO, where the index set of the variables is partitioned into two disjoint complementary sets~\cite{GT56,JRT92}. Associated with any instance of SOCO with strong duality, the optimal partition is uniquely defined by using solutions from the relative interior of the optimal set~\cite{BR2005,TW14}. Mathematically, given a fixed $\epsilon$, the \textit{optimal partition} of $(\mathrm{P}_{\epsilon})-(\mathrm{D}_{\epsilon})$ is defined as $\big\{\mathcal{B}(\epsilon),\mathcal{N}(\epsilon),\mathcal{R}(\epsilon),\mathcal{T}(\epsilon)\big\}$, where
\begin{align*}
\mathcal{B}(\epsilon)\!&:=\! \big\{i \mid \exists \ x(\epsilon) \in \mathcal{P}^*(\epsilon) \ \st \ x^i(\epsilon) \in \interior(\mathbb{L}^{n_i}_+)\big\},\\[-1\jot]
\mathcal{N}(\epsilon)\!&:= \!\big\{i \mid \exists \ \big(y(\epsilon);s(\epsilon)\big) \in \mathcal{D}^*(\epsilon) \ \st \ s^i(\epsilon) \in \interior(\mathbb{L}^{n_i}_+)\big\},\\[-1\jot]
\mathcal{R}(\epsilon)\!&:=\! \big\{i \mid \exists \ \big(x(\epsilon);y(\epsilon);s(\epsilon)\big) \in \mathcal{P}^*(\epsilon) \times \mathcal{D}^*(\epsilon) \ \st \ x^i(\epsilon),s^i(\epsilon) \in \boundary(\mathbb{L}^{n_i}_+)\setminus \{0\}\big\},\\[-1\jot]
\mathcal{T}(\epsilon)\!&:=\{1,\ldots,p\} \setminus \big\{\mathcal{B}(\epsilon) \cup \mathcal{N}(\epsilon) \cup \mathcal{R}(\epsilon)\big\}.
\end{align*}

\noindent
The convexity of the optimal set implies that the subsets $\mathcal{B}(\epsilon)$, $\mathcal{N}(\epsilon)$, $\mathcal{R}(\epsilon)$, and $\mathcal{T}(\epsilon)$ are mutually disjoint and their union is the index set $\{1,\ldots,p\}$. Additionally, it follows from the complementarity condition that for all $\big(x(\epsilon);y(\epsilon);s(\epsilon)\big) \in \mathcal{P}^*(\epsilon) \times \mathcal{D}^*(\epsilon)$ we have $x^i(\epsilon)=0$  for all $i \in \mathcal{N}(\epsilon)$ and $s^i(\epsilon)=0$ for all $i \in \mathcal{B}(\epsilon)$. 

\vspace{5px}
\noindent
We further partition $\mathcal{T}(\epsilon)$ into $\!\big\{\mathcal{T}_1(\epsilon),\mathcal{T}_2(\epsilon),\mathcal{T}_3(\epsilon)\big\}$, where
\begin{align*}
\mathcal{T}_1(\epsilon)\!&:=\!\big\{i \in \mathcal{T}(\epsilon) \mid x^i(\epsilon) = s^i(\epsilon)=0, \ \forall \big(x(\epsilon);y(\epsilon);s(\epsilon)\big) \in \mathcal{P}^*(\epsilon) \times \mathcal{D}^*(\epsilon) \big\}, \\[-1\jot]
\mathcal{T}_2(\epsilon)\!&:=\!\big\{i \in \mathcal{T}(\epsilon) \mid \exists \ x(\epsilon) \in \mathcal{P}^*(\epsilon) \ \st \ x^i(\epsilon) \in \boundary(\mathbb{L}^{n_i}_+)\setminus \{0\}, \\
 &\qquad\qquad\qquad\qquad\qquad\qquad\qquad\quad s^i(\epsilon) = 0, \ \forall\big(y(\epsilon);s(\epsilon)\big) \in \mathcal{D}^*(\epsilon)\big\},\\[-1\jot]
\mathcal{T}_3(\epsilon)\!&:=\!\big\{i \in \mathcal{T}(\epsilon) \mid \exists \ \big(y(\epsilon);s(\epsilon)\big) \in \mathcal{D}^*(\epsilon) \ \st \ s^i(\epsilon) \in \boundary(\mathbb{L}^{n_i}_+)\setminus \{0\}, \\
 &\qquad\qquad\qquad\qquad\qquad\qquad\qquad\qquad\qquad x^i(\epsilon)= 0, \ \forall  x(\epsilon) \in \mathcal{P}^*(\epsilon)  
\big\}.
\end{align*}

\noindent
The definition of maximally and strictly complementary optimal solutions can be rephrased by using the optimal partition of the problem: $\big(x(\epsilon);y(\epsilon);s(\epsilon)\big) \in \mathcal{P}^*(\epsilon) \times \mathcal{D}^*(\epsilon)$ is maximally complementary if and only if $x^i(\epsilon) \in \interior(\mathbb{L}^{n_i}_+)$ for all $i \in \mathcal{B}(\epsilon)$, $s^i(\epsilon) \in \interior(\mathbb{L}^{n_i}_+)$ for all $i \in \mathcal{N}(\epsilon)$, $x_1^i(\epsilon) > 0$ for all $i \in \mathcal{R}(\epsilon) \cup \mathcal{T}_2(\epsilon)$, and $s_1^i(\epsilon) > 0$ for all $i \in \mathcal{R}(\epsilon) \cup \mathcal{T}_3(\epsilon)$, see e.g.,~\cite[Proposition~1.3.3]{B09}. A maximally complementary optimal solution $\big(x^*(\epsilon);y^*(\epsilon);s^*(\epsilon)\big)$ is strictly complementary if and only if $\mathcal{T}(\epsilon)=\emptyset$.

\subsection{Nondegeneracy conditions}\label{nondegeneracy}
The concepts of primal and dual nondegeneracy for SOCO were introduced in~\cite{AG03}. Here, we tailor and adapt the nondegeneracy conditions only for a maximally complementary optimal solution. 

\vspace{5px}
\noindent
Assume that $\big(x^*(\epsilon);y^*(\epsilon);s^*(\epsilon)\big)$ is a maximally complementary optimal solution of $(\mathrm{P}_{\epsilon})$ and $(\mathrm{D}_{\epsilon})$ at a given $\epsilon$. Then $x^*(\epsilon)$ is called \textit{primal nondegenerate} if  
\begin{align}\label{primal_nondegenerate}
\begin{pmatrix} \big(A^i\bar{P}^{*i}(\epsilon)\big)_{\mathcal{R}(\epsilon) \cup \mathcal{T}_2(\epsilon)}, \ (A^i)_{\mathcal{B}(\epsilon)} \end{pmatrix}
\end{align}
has full row rank, where the columns of $\bar{P}^{*i}(\epsilon) \in \mathbb{R}^{n_i \times n_i - 1}$ are normalized eigenvectors of the positive eigenvalues of $L\big((x^{*i}(\epsilon)\big)$. Furthermore, $\big(y^*(\epsilon);s^*(\epsilon)\big)$ is called \textit{dual nondegenerate} if
\begin{align*}
\begin{pmatrix} \big(A^iR^i s^{*i}(\epsilon)\big)_{\mathcal{R}(\epsilon) \cup \mathcal{T}_3(\epsilon)}, \ (A^i)_{\mathcal{B}(\epsilon) \cup \mathcal{T}_1(\epsilon) \cup \mathcal{T}_2(\epsilon)} \end{pmatrix}
\end{align*}
has full column rank, where $R^i$ is defined in~\eqref{R_definition}. Given a fixed $\epsilon$, if there exists a primal (dual) nondegenerate optimal solution, then the dual (primal) optimal set mapping is single-valued at $\epsilon$. Furthermore, if there exists a strictly complementary optimal solution at $\epsilon$, then the reverse direction is true as well. The proof can be found in~\cite{AG03}.

\begin{remark}
In this paper, the primal and dual nondegeneracy conditions are said to hold at $\epsilon$ if there exists a nondegenerate maximally complementary optimal solution at~$\epsilon$.
\end{remark}
We invoke the primal and dual nondegeneracy conditions in Section~\ref{nonlinearity_interval_strict_fails} for the identification of a transition point.

\subsection{Set-valued analysis}\label{set_valued_analysis}
In this section, we briefly review the continuity of set-valued mappings from~\cite{RD14,Rock09}. Let $\mathbb{N}$ be the set of natural numbers, $\mathcal{J}$ be the collection of subsets $J \subset \mathbb{N}$ with $\mathbb{N} \setminus J$ being finite, $\mathcal{J}_{\infty}$ denote the collection of all infinite subsets of $\mathbb{N}$, and $\{\mathcal{C}_k\}_{k=1}^{\infty}$ be a sequence of subsets of $\mathbb{R}^{\bar{n}}$. The \textit{outer limit} of $\{\mathcal{C}_k\}_{k=1}^{\infty}$ is defined as
\begin{align*}
\limsup\limits_{k \to \infty} \mathcal{C}_k\!:=\!\big\{x \mid \exists \ J \in \mathcal{J}_{\infty} \ \text{and} \ x_k \in \mathcal{C}_k \ \text{for} \ k \in J \ \st \ \lim_{k \in J} x_k= x \big\},
\end{align*}
where $\lim_{k \in J} x_k$ denotes the limit of a convergent sequence $x_k$ as $k \to \infty$ and $k \in J$. On the other hand, the \textit{inner limit} of $\{\mathcal{C}_k\}_{k=1}^{\infty}$ is given by
\begin{align*}
\liminf\limits_{k \to \infty} \mathcal{C}_k\!:=\!\big\{x \mid \exists \ J \in \mathcal{J} \ \text{and} \ x_k \in \mathcal{C}_k \ \text{for} \ \ k \in J \ \st \ \lim_{k \in J} x_k= x \big\}.
\end{align*}
If the inner and outer limits coincide, then the limit of $\{\mathcal{C}_k\}_{k=1}^{\infty}$ exists \textit{in the sense of Painlev\'e-Kuratowski}, i.e.,
\begin{align*}
\lim_{k \to \infty} \mathcal{C}_k : =\limsup\limits_{k \to \infty} \mathcal{C}_k = \liminf\limits_{k \to \infty} \mathcal{C}_k.
\end{align*}
When $\mathcal{C}_k \neq \emptyset$, $\limsup\limits_{k \to \infty} \mathcal{C}_k$ denotes the collection of all accumulation points of $\{x_k\}_{k=1}^{\infty}$ such that $x_k \in \mathcal{C}_k$, while $\liminf\limits_{k \to \infty} \mathcal{C}_k$ represents the collection of all limit points of $\{x_k\}_{k=1}^{\infty}$. Recall that both the $\limsup$ and $\liminf$ of a sequence of sets are closed subsets of $\mathbb{R}^{\bar{n}}$~\cite[Section 3.1]{RD14}.

\vspace{5px}
\noindent
A \textit{set-valued} mapping $\mathrm{\Phi}(\epsilon):\mathbb{R} \rightrightarrows \mathbb{R}^{\bar{n}}$ assigns a subset of $\mathbb{R}^{\bar{n}}$ to each element of $\epsilon \in \mathbb{R}$. The domain of the set-valued mapping $\mathrm{\Phi}(\epsilon)$ is defined as
\begin{align*}
\dom(\mathrm{\Phi})\!:=\!\big\{\epsilon \in \mathbb{R} \mid \mathrm{\Phi}(\epsilon) \neq \emptyset \big\},
\end{align*}
and its range is given by 
\begin{align*}
\mathrm{range}(\mathrm{\Phi})\!:=\!\big\{x \in \mathbb{R}^{\bar{n}} \mid x \in \mathrm{\Phi}(\epsilon), \ \text{for some} \ \epsilon \in \mathbb{R} \big\}=\bigcup\limits_{\epsilon \in \mathbb{R}} \mathrm{\Phi}(\epsilon).
\end{align*}
\noindent
Various forms of continuity exist for a set-valued mapping. In this paper, continuity of a set-valued mapping is formed on the basis of Painlev\'e-Kuratowski set convergence, see~\cite[Chapter~5]{Rock09}, which is equivalent to the notion of continuity of a \textit{point-to-set map} in~\cite{Ho73b}, see~\cite[Corollary~1.1 and Theorem~2]{Ho73b} and~\cite[Theorem~5.7]{Rock09}.

\vspace{5px}
\noindent
Let $\mathrm{\Gamma}$ be a subset of $\mathbb{R}$ containing $\bar{\epsilon}$, and define
\begin{align*}
\limsup\limits_{\epsilon \to \bar{\epsilon}} \mathrm{\Phi}(\epsilon)&\!:=\!\big \{x \mid \exists \{\epsilon_k\}_{k=1}^\infty \subseteq \mathrm{\Gamma} \ \text{with} \ \epsilon_k \to \bar{\epsilon}, \exists \ x_k \to x \ \text{with} \ x_k \in \mathrm{\Phi}(\epsilon_k) \big \},\\[-1\jot]
\liminf\limits_{\epsilon \to \bar{\epsilon}} \mathrm{\Phi}(\epsilon)&\!:=\!\big \{x \mid \forall \{\epsilon_k\}_{k=1}^\infty \subseteq \mathrm{\Gamma} \ \text{with} \ \epsilon_k \to \bar{\epsilon}, \exists \ x_k \to x \ \text{with} \ x_k \in \mathrm{\Phi}(\epsilon_k) \big \}.
\end{align*}
Then a set-valued mapping $\mathrm{\Phi}(\epsilon)$ is called \textit{outer semicontinuous} at $\bar{\epsilon}$ relative to $\mathrm{\Gamma}$ if 
\begin{align*}
\limsup_{\epsilon \to \bar{\epsilon}} \mathrm{\Phi}(\epsilon) \subseteq \mathrm{\Phi}(\bar{\epsilon}) 
\end{align*}
and \textit{inner semicontinuous} at $\bar{\epsilon}$ relative to $\mathrm{\Gamma}$ if
\begin{align*}
\liminf_{\epsilon\to \bar{\epsilon}} \mathrm{\Phi}(\epsilon) \supseteq \mathrm{\Phi}(\bar{\epsilon})
\end{align*}
holds. The set-valued mapping $\mathrm{\Phi}(\epsilon)$ is \textit{Painlev\'e-Kuratowski continuous} at $\bar{\epsilon}$ relative to $\mathrm{\Gamma}$ if it is both outer and inner semicontinuous at $\bar{\epsilon}$ relative to $\mathrm{\Gamma}$.  

\vspace{5px}
\noindent
By Assumptions~\ref{LID} and~\ref{IPC}, we can show that $\mathcal{P}^*(\epsilon)$ and $\mathcal{D}^*(\epsilon)$ are outer semicontinuous relative to $\interior(\mathcal{E})$. The result follows from~\cite[Theorems~8 and~12]{Ho73b}.

\begin{lemma}\label{optimal_set_mapping_outer_semicontinuity}
The set-valued mappings $\mathcal{P}^*(\epsilon)$ and $\mathcal{D}^*(\epsilon)$ are outer semicontinuous relative to $\interior(\mathcal{E})$.
\end{lemma}
 
\noindent
The optimal set mapping may fail to be inner semicontinuous relative to $\interior(\mathcal{E})$, e.g., when either the primal or the dual nondegeneracy condition fails at $\epsilon$, while the strict complementarity condition holds. Nevertheless, sufficient conditions can be given for the continuity of $\mathcal{P}^*(\epsilon)$ and $\mathcal{D}^*(\epsilon)$, regardless of the nondegeneracy conditions. First, it is easy to show, under Assumptions~\ref{LID} and~\ref{IPC}, that the optimal set mapping is \textit{locally bounded}%
\footnote{Here, local boundedness is equivalent to uniform compactness in~\cite{Ho73b}.} at $\bar{\epsilon} \in \interior(\mathcal{E})$~\cite[Lemma~3.2]{SW20}, i.e., there exist $\varsigma > 0$ and bounded sets $\mathcal{C}_1$ and $\mathcal{C}_2$ such that 
\begin{align}\label{optimal_set_locally_bounded}
\bigcup_{\epsilon \in (\bar{\epsilon}-\varsigma, \bar{\epsilon} + \varsigma)} \mathcal{P}^*(\epsilon) \subseteq \mathcal{C}_1 \quad \text{and} \quad \bigcup_{\epsilon \in (\bar{\epsilon}-\varsigma, \bar{\epsilon} + \varsigma)} \mathcal{D}^*(\epsilon) \subseteq \mathcal{C}_2.
\end{align}
Then the continuity follows from the uniqueness condition.
\begin{lemma}\label{continuity_sufficient_condition}
If $\mathcal{P}^*(\epsilon)$ is single-valued at $\bar{\epsilon} \in \interior(\mathcal{E})$, then $\mathcal{P}^*(\epsilon)$ is continuous at $\bar{\epsilon}$ relative to $\interior(\mathcal{E})$. An analogous result holds for $\mathcal{D}^*(\epsilon)$. 
\end{lemma}
\begin{proof}
The proof is immediate from Lemma~\ref{optimal_set_mapping_outer_semicontinuity}, local boundedness of $\mathcal{P}^*(\epsilon)$ and $\mathcal{D}^*(\epsilon)$ at $\bar{\epsilon}$, and~\cite[Corollary~8.1]{Ho73b}. \qed
\end{proof}
Even though the primal or dual optimal set mapping is not necessarily inner semicontinuous relative to $\interior(\mathcal{E})$, the set of points at which $\mathcal{P}^*(\epsilon)$ and $\mathcal{D}^*(\epsilon)$ fail to be continuous relative to $\interior(\mathcal{E})$ is proven to be the union of countably many nowhere dense subsets of $\interior(\mathcal{E})$, i.e., it is of \textit{first category} in $\interior(\mathcal{E})$. This is the consequence of Lemma~\ref{optimal_set_mapping_outer_semicontinuity} and~\cite[Theorem~5.55]{Rock09}. Then the following result is in order.

\begin{proposition}[Theorem~1.3 in~\cite{Ox80}]
The set of points at which $\mathcal{P}^*(\epsilon)$ and $\mathcal{D}^*(\epsilon)$ are continuous is dense in $\interior(\mathcal{E})$. 
\end{proposition}

\begin{remark}
From this point on, unless stated otherwise, by the inner/outer semicontinuity of $\mathcal{P}^*(\epsilon)$ and $\mathcal{D}^*(\epsilon)$ at a given $\epsilon \in \interior(\mathcal{E})$ we mean inner/outer semicontinuity at $\epsilon$ relative to $\interior(\mathcal{E})$.
\end{remark}

\vspace{5px}
\noindent
The continuity results are used in Sections~\ref{sec:continuity_nonlinearity} and~\ref{nonlinearity_interval_strict_fails} for the identification of a nonlinearity interval and a transition point.

\section{Sensitivity of the optimal partition}\label{sec:sensitivity_optimal_partition}
In~\cite{MT20}, the notion of a nonlinearity interval and a transition point was formally introduced for the optimal partition of a parametric SDO problem. In this section, we introduce those notions for the optimal partition of $(\mathrm{P}_{\epsilon})-(\mathrm{D}_{\epsilon})$, which is defined on the basis of a different algebraic structure, see Section~\ref{conclusion}. From now on, the optimal partition of $(\mathrm{P}_{\epsilon})$ and $(\mathrm{D}_{\epsilon})$ at a given $\epsilon$ is denoted by $\pi(\epsilon)\!:=\!\big\{\mathcal{B}(\epsilon),\mathcal{N}(\epsilon),\mathcal{R}(\epsilon),\mathcal{T}(\epsilon)\big\}$. 

\subsection{Invariancy sets, nonlinearity intervals, and transition points}\label{sec:nonlinearity_transition}
For parametric LO and LCQO problems, the interval $\interior(\mathcal{E})$ can be entirely partitioned into invariancy sets, on which the optimal partition remains unchanged w.r.t. $\epsilon$~\cite{BJRT96,JRT92}. An invariancy set can be analogously defined for $(\mathrm{P}_{\epsilon})-(\mathrm{D}_{\epsilon})$. This definition is in accordance with~\cite[Section 4]{Yil04} for a linear conic optimization problem. 
\begin{definition}~\label{invariancy_interval}
An \textit{invariancy set} $\mathcal{E}_{\mathrm{inv}}$ is a maximal subset of $\interior(\mathcal{\mathcal{E}})$ such that $\pi(\epsilon')=\pi(\epsilon'')$ for all $\epsilon',\epsilon'' \in \mathcal{E}_{\mathrm{inv}}$, and the extreme rays $\mathbb{R}_+ x^{*i}(\epsilon)$ for $i \in \mathcal{R}(\epsilon) \cup \mathcal{T}_2(\epsilon)$ and $\mathbb{R}_+ s^{*i}(\epsilon)$ for $i \in \mathcal{R}(\epsilon) \cup \mathcal{T}_3(\epsilon)$ are invariant w.r.t. $\epsilon \in \mathcal{E}_{\mathrm{inv}}$, where $\big(x^*(\epsilon);y^*(\epsilon);s^*(\epsilon)\big)$ is any maximally complementary optimal solution at $\epsilon$. If it is not a singleton, then $\mathcal{E}_{\mathrm{inv}}$ is called an \textit{invariancy interval}.
\end{definition}
We remark here that the notion of an invariancy set is well-defined, i.e., it is independent of the choice of a maximally complementary optimal solution.
\begin{proposition}\label{invariancy_of_extreme_ray}
At a given $\epsilon$, let $\big(x^{*}(\epsilon);y^{*}(\epsilon);s^{*}(\epsilon)\big)$ and $\big(\check{x}^{*}(\epsilon);\check{y}^{*}(\epsilon);\check{s}^{*}(\epsilon)\big)$ be two arbitrary maximally complementary optimal solutions. Then it holds that 
\begin{align*}
\mathbb{R}_+ x^{*i}(\epsilon) &= \mathbb{R}_+ \check{x}^{*i}(\epsilon), & \forall i &\in \mathcal{R}(\epsilon) \cup \mathcal{T}_2(\epsilon),\\[-1\jot]
\mathbb{R}_+ s^{*i}(\epsilon) &= \mathbb{R}_+ \check{s}^{*i}(\epsilon), & \forall i &\in \mathcal{R}(\epsilon) \cup \mathcal{T}_3(\epsilon).
\end{align*}
\end{proposition}
\begin{proof}
If $\mathbb{R}_+ x^{*i}(\epsilon) \neq \mathbb{R}_+ \check{x}^{*i}(\epsilon)$ for an $i \in \mathcal{R}(\epsilon) \cup \mathcal{T}_2(\epsilon)$, then by the convexity of the optimal set and the triangle inequality we would have $\gamma x^{*}(\epsilon) + (1-\gamma) \check{x}^{*}(\epsilon) \in \mathcal{P}^*(\epsilon)$ for every $\gamma \in(0,1)$ such that $\gamma x^{*i}(\epsilon) + (1-\gamma) \check{x}^{*i}(\epsilon) \in \interior(\mathbb{L}^{n_i}_+)$. However, the latter would imply that $i \in \mathcal{B}(\epsilon)$, contradicting the assumption. The proof is analogous for an extreme ray $\mathbb{R}_+ s^{*i}(\epsilon)$. \qed
\end{proof}

\vspace{5px}
\noindent
As in parametric LO and SDO problems, an invariancy interval is open, and the primal optimal set is invariant w.r.t. $\epsilon$ on an invariancy interval. The proof is a word by word specialization from the SDO case, and is omitted for the sake of brevity, see e.g.,~\cite[Lemma~3.3 and Remark~3.1]{MT20} for details. The boundary points of an invariancy set can be obtained by solving a pair of auxiliary SOCO problems, see~\cite[Section 4]{Yil04}. Thus, a singleton invariancy set is identified when the boundary points from the auxiliary SOCO problems are identical.

\vspace{5px}
\noindent
It is easy to see that a singleton invariancy set $\{\bar{\epsilon}\}$ exists, i.e., when either the optimal partition $\pi(\epsilon)$, or the extreme rays $\mathbb{R}_+ x^{*i}(\epsilon)$ and $\mathbb{R}_+ s^{*i}(\epsilon)$ for some $i \in \mathcal{R}(\epsilon) \cup \mathcal{T}_2(\epsilon) \cup \mathcal{T}_3(\epsilon)$, or both change in every neighborhood of $\bar{\epsilon}$. However, unlike parametric LO and LCQO problems, infinitely many singleton invariancy sets may exist for $(\mathrm{P}_{\epsilon})-(\mathrm{D}_{\epsilon})$, as demonstrated by the following parametric problem%
\footnote{See~\cite[Example~3.1]{MT20} for an instance of a parametric SDO problem with infinitely many singleton invariancy sets.}:
\begin{equation}\label{ex:introductory_SOCO}
\begin{aligned}
\min \quad  &-\epsilon x^1_2 - (1-\epsilon) x^1_3\\[-1\jot]
\st \quad& x^1_1 = 1,\\[-1\jot]
&x^1_3 - x^2_1 = 0,\\[-1\jot]
&x^1_2-x^2_2 = 1,\\[-1\jot]
&x^1_1 \ge \sqrt{(x^1_2)^2 + (x^1_3)^2},\\[-1\jot]
&x^2_1 \ge |x^2_2|,
\end{aligned}
\end{equation}
where $\mathcal{E}=\mathbb{R}$. One can check that the optimal partition on $\mathbb{R}$ is given by
\begin{align*}
\big\{\mathcal{B}(\epsilon),\mathcal{N}(\epsilon),\mathcal{R}(\epsilon),\mathcal{T}(\epsilon)\big\}=\begin{cases} \big\{\emptyset,\emptyset,\{1,2\},\{\emptyset,\emptyset,\emptyset\}\big\}, \ \ &\epsilon \in (-\infty, 0),\\[-.5\jot]
\big\{\emptyset,\emptyset,\{1\},\{\emptyset,\{2\},\emptyset\}\big\},  &\epsilon = 0,\\[-.5\jot]
\big\{\{2\},\emptyset,\{1\},\{\emptyset,\emptyset,\emptyset\}\big\},  &\epsilon \in (0,1),\\[-.5\jot]
\big\{\emptyset,\emptyset,\{1\},\{\{2\},\emptyset,\emptyset\}\big\},  &\epsilon = 1,\\[-.5\jot]
\big\{\emptyset,\{2\},\{1\},\{\emptyset,\emptyset,\emptyset\}\big\},  &\epsilon \in (1,\infty),
\end{cases}
\end{align*}   
where
\begin{align*}
x^{*1}(\epsilon)=\begin{pmatrix} 1\\ \frac{\epsilon}{\sqrt{(1-\epsilon)^2+\epsilon^2}} \\ \frac{1-\epsilon}{\sqrt{(1-\epsilon)^2 + \epsilon^2}} \end{pmatrix}, \ s^{*1}(\epsilon)=\begin{pmatrix} \sqrt{(1-\epsilon)^2+\epsilon^2}\\ -\epsilon \\ \epsilon-1 \end{pmatrix}, \qquad \forall \epsilon \in (0,1). 
\end{align*}   
Notice that the extreme rays $\mathbb{R}_+ x^{*1}(\epsilon)$ and $\mathbb{R}_+ s^{*1}(\epsilon)$ vary continuously with $\epsilon$ while $\pi(\epsilon)$ is fixed on the interval $(0,1)$, see Figure~\ref{Circle_SOCO_example}. In this case, $(0,1)$ is called a nonlinearity interval, and $\{0,1\}$ denotes the set of transition points.

\begin{figure}
 \begin{minipage}[c]{0.60\textwidth}
\includegraphics[height=1.8in]{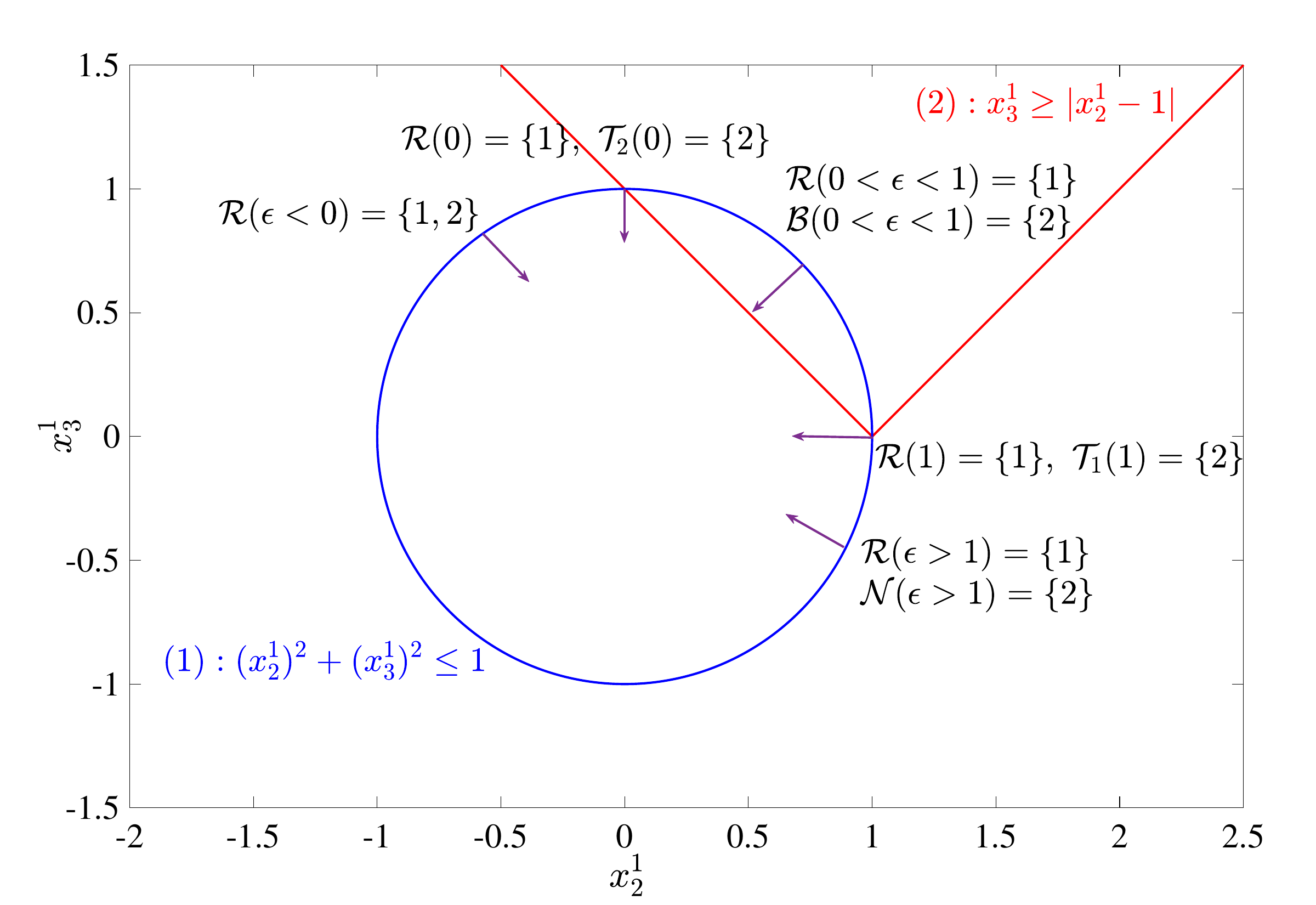}
\end{minipage}\hfill
\begin{minipage}[c]{0.4\textwidth}
\caption{The projection of the feasible region of~\eqref{ex:introductory_SOCO} onto $(x^1_2,x^1_3)$.}
\label{Circle_SOCO_example}
\end{minipage}
\end{figure}

\vspace{5px}
\noindent
Now, we formally define the notions of nonlinearity interval and transition point.
\begin{definition}\label{nonlinearity_SOCO}
A \textit{nonlinearity interval} is a non-singleton open maximal subinterval $\mathcal{E}_{\mathrm{non}}$ of $\interior(\mathcal{E})$ such that $\pi(\epsilon')=\pi(\epsilon'')$ for any two $\epsilon',\epsilon'' \in \mathcal{E}_{\mathrm{non}}$, while $\mathbb{R}_+ x^{*i}(\epsilon)$ for some $i \in \mathcal{R}(\epsilon) \cup \mathcal{T}_2(\epsilon)$ or $\mathbb{R}_+ s^{*i}(\epsilon)$ for some $i \in \mathcal{R}(\epsilon) \cup \mathcal{T}_3(\epsilon)$ varies with $\epsilon$. 
\end{definition} 
Notice that a nonlinearity interval is well-defined by Proposition~\ref{invariancy_of_extreme_ray}. Furthermore, the eigenstructures of $L\big(x^{*i}(\epsilon)\big)$ and $L\big(s^{*i}(\epsilon)\big)$, see~\eqref{rank_structure}, imply that ranks of $L\big(x^{*}(\epsilon)\big)$ and $L\big(s^{*}(\epsilon)\big)$ remain constant on a nonlinearity interval, where $L(x)$ and $L(s)$ are defined by~\eqref{blk_diag_arrow_hat}. Obviously, if $\mathcal{R}(\epsilon) = \mathcal{T}_2(\epsilon) = \mathcal{T}_3(\epsilon)=\emptyset$ on $\interior(\mathcal{E})$, then no nonlinearity interval exists.

\begin{definition}\label{transition_SOCO}
A singleton invariancy set $\{\bar{\epsilon}\} \in \interior(\mathcal{E})$ is called a \textit{transition point} if for every $\varsigma > 0$ there exists an $\epsilon \in (\bar{\epsilon}-\varsigma,\bar{\epsilon}+\varsigma) \subseteq \interior(\mathcal{E})$ such that $\pi(\bar{\epsilon}) \neq \pi(\epsilon)$. 
\end{definition}

\noindent
It can be deducted from Definitions~\ref{nonlinearity_SOCO} and~\ref{transition_SOCO} that a singleton invariancy set either belongs to a nonlinearity interval, or it is a transition point. Further, it immediately follows that the boundary points, in $\interior(\mathcal{E})$, of invariancy or nonlinearity intervals belong to the set of transition points. On the other hand, a \textit{semi-algebraic}~\cite[Page~57]{BPR06} formulation of the optimal set reveals that a transition point must be a boundary point, in $\interior(\mathcal{E})$, of an invariancy or a nonlinearity interval, as stated in Theorem~\ref{transition_point_finiteness}. As a result, one can partition $\interior(\mathcal{E})$ into the finite union of invariancy intervals, nonlinearity intervals, and transition points.

\begin{theorem}\label{transition_point_finiteness}
The set of transition points is finite.
\end{theorem}
\begin{proof}
Given a fixed optimal partition $\pi(\bar{\epsilon})$, the set of all $\epsilon$ with the optimal partition $\pi(\bar{\epsilon})$ can be formulated as 
\begin{align*}
S_{\pi(\bar{\epsilon})}\!:=\!\Big\{\epsilon \mid \exists \ (x;y;s) \in \ri\!\big(\mathcal{P}^*(\epsilon) \times \mathcal{D}^*(\epsilon)\big), \\[-1.5\jot]
(x_1^i)^2-\|x^i_{2:n_i}\|^2 &> 0, & i &\in \mathcal{B}(\bar{\epsilon}),\\[-1.5\jot]
(x_1^i)^2-\|x^i_{2:n_i}\|^2 &= 0, & i &\in  \mathcal{T}_2(\bar{\epsilon}),\\[-1.5\jot]
(s_1^i)^2-\|s^i_{2:n_i}\|^2 &>0, & i &\in \mathcal{N}(\bar{\epsilon}),\\[-1.5\jot]
(s_1^i)^2-\|s^i_{2:n_i}\|^2 &= 0, & i &\in  \mathcal{T}_3(\bar{\epsilon}),\\[-1.5\jot]
x_1^i &> 0, & i &\in \mathcal{R}(\bar{\epsilon}) \cup \mathcal{T}_2(\bar{\epsilon}),\\[-1.5\jot]
x^i &= 0, & i &\in \mathcal{T}_1(\bar{\epsilon}) \cup \mathcal{T}_3(\bar{\epsilon}),\\[-1.5\jot]
s_1^i &> 0, & i &\in \mathcal{R}(\bar{\epsilon}) \cup \mathcal{T}_3(\bar{\epsilon}),\\[-1.5\jot]
s^i &= 0, & i &\in \mathcal{T}_1(\bar{\epsilon}) \cup \mathcal{T}_2(\bar{\epsilon})\Big\}.
\end{align*} 
Observe that $S_{\pi(\bar{\epsilon})}$ is a semi-algebraic subset of $\mathbb{R}$, being the projection of a semi-algebraic set defined by a boolean combination of polynomial equalities and inequalities~\cite[Theorem~2.76]{BPR06}. Note that $S_{\pi(\bar{\epsilon})}$ might be empty or disconnected in $\mathbb{R}$. Since a semi-algebraic set has a finite number of connected components~\cite[Theorem~5.22]{BPR06}, and the boundary points, in $\interior(\mathcal{E})$, of the components are transition points, the set of transition points with a fixed optimal partition $\pi(\bar{\epsilon})$ is finite. Then the result follows by noting that $\pi(.)$ can only take a finite number of possibilities. \qed
\end{proof}

\vspace{5px}
\noindent
Notice that the connected components of $S_{\pi(\bar{\epsilon})}$ in Theorem~\ref{transition_point_finiteness} are either points or intervals~\cite[Corollary~2.79]{BPR06}: an isolated point in $S_{\pi(\bar{\epsilon})}$ is a transition point, whereas the interior of an interval is either an invariancy or a nonlinearity interval. Conversely, an invariancy interval with the optimal partition $\pi(\bar{\epsilon})$ is indeed an open connected component of $S_{\pi(\bar{\epsilon})}$ in $\mathbb{R}$. It is unknown, however, if the component containing a nonlinearity interval $\mathcal{E}_{\mathrm{non}}$ is open in $\mathbb{R}$. See ~\cite[Corollary~3.8]{MT20}, which can be also specialized for $(\mathrm{P}_{\epsilon})-(\mathrm{D}_{\epsilon})$.

\begin{remark}
In case that $\mathcal{R}(\epsilon) \cup \mathcal{T}_2(\epsilon) \cup \mathcal{T}_3(\epsilon) \neq \emptyset$, any two nonlinearity intervals or a nonlinearity interval and a transition point might have the same optimal partition. This is in contrast to parametric LO and LCQO problems where the invariancy sets are associated with distinct optimal partitions~\cite{BJRT96,JRT92}. For instance, consider the optimal partition of the following parametric SOCO problem:
\begin{equation}\label{optim_value_analytic}
\begin{aligned}
\min  \ \ &(1-2\epsilon) x^1_2 - x^1_3\\[-1\jot]
\st \ \  &x^1_1 = 1,\\[-1\jot]
&x^2_1 = 2,\\[-1\jot]
&x^2_2 - x^1_2 = 0,\\[-1\jot]
&x^2_3-2x^1_3 = 0,\\[-1\jot]
&x^1_1 \ge \sqrt{(x^1_2)^2 + (x^1_3)^2},\\[-1\jot]
&x^2_1 \ge \sqrt{(x^2_2)^2 + (x^2_3)^2},
\end{aligned}
\end{equation}
where the optimal set is given by
\begin{align*}
x^*(\epsilon)&=\begin{pmatrix} 1 \\ \frac{2\epsilon-1}{\sqrt{4\epsilon^2-4\epsilon+2}} \\ \frac{1}{\sqrt{4\epsilon^2-4\epsilon+2}} \\ 2 \\ \frac{2\epsilon-1}{\sqrt{4\epsilon^2-4\epsilon+2}} \\ \frac{2}{\sqrt{4\epsilon^2-4\epsilon+2}} \end{pmatrix}, & s^*(\epsilon)&=\begin{pmatrix} \sqrt{4\epsilon^2-4\epsilon+2} \\ 1-2\epsilon \\ -1 \\ 0 \\ 0 \\ 0 \end{pmatrix}, \quad  \epsilon \in (-\infty,\tfrac12) \cup \ (\tfrac12,\infty),\\
x^*(\tfrac12)&=\begin{pmatrix} 1, 0, 1, 2, 0, 2 \end{pmatrix}^T, & s^*(\tfrac12)&=\begin{pmatrix} \gamma, 0, -\gamma, \frac{1-\gamma}{2}, 0, \frac{\gamma-1}{2} \end{pmatrix}^T, \quad  \gamma \in [0,1].
\end{align*}
In this case, the two nonlinearity intervals $(-\infty,\frac12)$ and $(\frac12,\infty)$ with identical optimal partitions are separated by a transition point at $\epsilon = \frac12$, see Figure~\ref{Circle_Ellipse_example}.
\begin{figure}
\center
\includegraphics[height=1.8in]{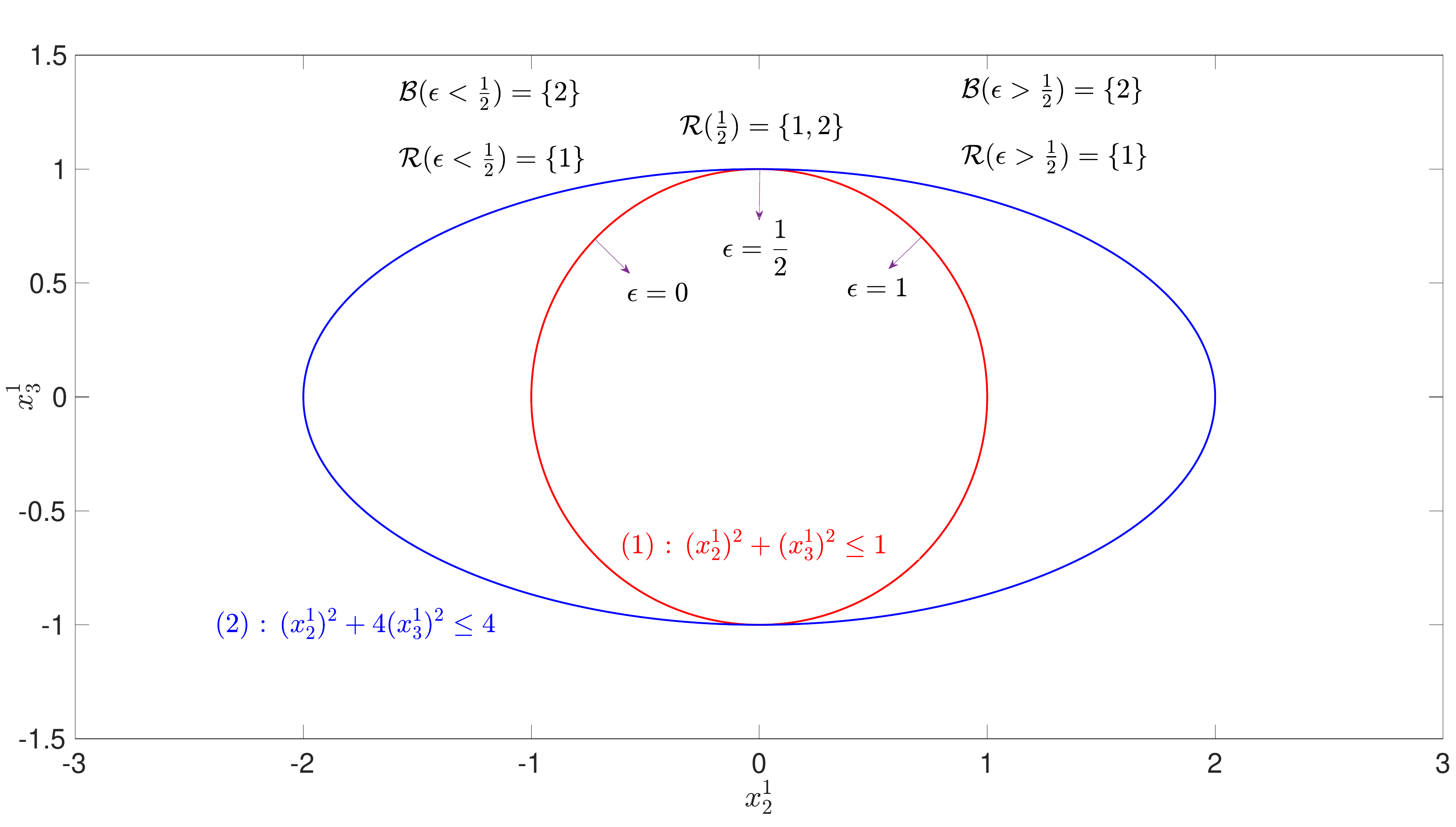}
\caption{The projection of the feasible region of~\eqref{optim_value_analytic} onto $(x^1_2,x^1_3)$.}
\label{Circle_Ellipse_example}
\end{figure}
\end{remark}
\setlength{\abovedisplayskip}{1pt}
\setlength{\belowdisplayskip}{1pt}
\subsection{On the identification of a nonlinearity interval}\label{sec:continuity_nonlinearity}
Recall from Definition~\ref{nonlinearity_SOCO} that both the primal and dual optimal sets vary with $\epsilon$ on a nonlinearity interval. As problem~\eqref{ex:introductory_SOCO} indicates%
\footnote{For this problem, both the primal and dual optimal set mappings are continuous at the transition point $\epsilon = 0$.}, solely continuity of the optimal set mapping at $\bar{\epsilon}$ does not induce the existence of a nonlinearity interval. The following result is important enough to be stated as a lemma.
\begin{lemma}\label{partition_inclusion_lemma}
Let both $\mathcal{P}^*(\epsilon)$ and $\mathcal{D}^*(\epsilon)$ be continuous at $\bar{\epsilon}$. Then, for all $\epsilon$ in a sufficiently small neighborhood of $\bar{\epsilon}$, we have
\begin{align}\label{inclusion_results}
\mathcal{B}(\bar{\epsilon}) \subseteq \mathcal{B}(\epsilon), \quad \mathcal{N}(\bar{\epsilon}) \subseteq \mathcal{N}(\epsilon), \quad \mathcal{R}(\bar{\epsilon}) \subseteq \mathcal{R}(\epsilon),
\end{align}
which also implies $\mathcal{T}(\epsilon) \subseteq \mathcal{T}(\bar{\epsilon})$. Furthermore, we have
\begin{align*}
\mathcal{T}_2(\bar{\epsilon}) \cap \big\{\mathcal{N}(\epsilon) \cup \mathcal{T}_1(\epsilon) \cup \mathcal{T}_3(\epsilon)\big\} &=\emptyset,\\[-1\jot]
 \mathcal{T}_3(\bar{\epsilon}) \cap \big\{\mathcal{B}(\epsilon) \cup \mathcal{T}_1(\epsilon) \cup \mathcal{T}_2(\epsilon)\big\}&=\emptyset.
\end{align*}
\end{lemma}
\begin{proof}
Let $\big(x^*(\bar{\epsilon});y^*(\bar{\epsilon});s^*(\bar{\epsilon})\big)$ be a maximally complementary optimal solution at $\bar{\epsilon}$ and $U \subseteq \mathbb{R}^{m+2\bar{n}}$ be a neighborhood of $\big(x^*(\bar{\epsilon});y^*(\bar{\epsilon});s^*(\bar{\epsilon})\big)$ in the Euclidean topology. Then $U$ can be made so small that for all $(x;y;s) \in U$ we have
\begin{align}\label{neighborhood_continuity_condition}
\begin{cases}
 x_1^i-\|x_{2:n_i}^i\| > 0,  &\forall \  i \in \mathcal{B}(\bar{\epsilon}),\\[-1\jot]
 s_1^i-\|s_{2:n_i}^i\| > 0,  &\forall \  i \in \mathcal{N}(\bar{\epsilon}),\\[-1\jot]
 x_1^i, s_1^i> 0,  &\forall \  i \in \mathcal{R}(\bar{\epsilon}),\\[-1\jot]
 x_1^i> 0,  &\forall \  i \in \mathcal{T}_2(\bar{\epsilon}),\\[-1\jot]
 s_1^i> 0,  &\forall \  i \in \mathcal{T}_3(\bar{\epsilon}).
 \end{cases}
\end{align} 
Now, by inner semicontinuity of $\mathcal{P}^*(\epsilon)$ and $\mathcal{D}^*(\epsilon)$ at $\bar{\epsilon}$~\cite[Theorem~3B.2(b)]{RD14}, there exist a neighborhood $V$ of $\bar{\epsilon}$ and an optimal $\big(x(\epsilon);y(\epsilon);s(\epsilon)\big) \in U$ for every $\epsilon \in V$. By the definition of the optimal partition, all this implies~\eqref{inclusion_results} and $\mathcal{T}(\epsilon) \subseteq \mathcal{T}(\bar{\epsilon})$. Finally, notice from~\eqref{neighborhood_continuity_condition} that $x_1^{i}(\epsilon) > 0$ for every $i \in \mathcal{T}_2(\bar{\epsilon})$ and every $\epsilon \in V$, while $x_1^{i}(\epsilon)$ must be $0$ for any $i \in \mathcal{N}(\epsilon) \cup \mathcal{T}_1(\epsilon) \cup \mathcal{T}_3(\epsilon)$. The same continuity argument suffices for $\mathcal{T}_3(\bar{\epsilon})$, which completes the proof. \qed
\end{proof} 

\vspace{5px}
\noindent
As a result, continuity in the presence of the strict complementarity condition becomes sufficient for the existence of a nonlinearity interval. 
\begin{theorem}\label{nonlinearity_continuity_SOCO}
Let $\{\bar{\epsilon}\}$ be a singleton invariancy set. If $\big(x^*(\bar{\epsilon});y^*(\bar{\epsilon});s^*(\bar{\epsilon})\big)$ is strictly complementary and both $\mathcal{P}^*(\epsilon)$ and $\mathcal{D}^*(\epsilon)$ are continuous at $\bar{\epsilon}$, then $\bar{\epsilon}$ belongs to a nonlinearity interval. 
\end{theorem}
\begin{proof}
Since $\mathcal{T}(\bar{\epsilon}) =\emptyset$, it follows from Lemma~\ref{partition_inclusion_lemma} that $\mathcal{T}(\epsilon) \subseteq \mathcal{T}(\bar{\epsilon}) = \emptyset$, and thus $\mathcal{B}(\bar{\epsilon}) = \mathcal{B}(\epsilon)$, $\mathcal{N}(\bar{\epsilon}) = \mathcal{N}(\epsilon)$, and $\mathcal{R}(\bar{\epsilon}) = \mathcal{R}(\epsilon)$ for all $\epsilon$ in a sufficiently small neighborhood of $\bar{\epsilon}$. Since $\{\bar{\epsilon}\}$ is a singleton invariancy set, it must belong to a nonlinearity interval.  \qed
\end{proof}

\vspace{5px}
\noindent
Theorem~\ref{nonlinearity_continuity_SOCO} does not yield a complete characterization for the existence of a nonlinearity interval, because either strict complementarity or continuity might fail on a nonlinearity interval. All this makes the identification and computation of a nonlinearity interval a nontrivial task. For instance, $\mathcal{P}^*(\epsilon)$ fails to be continuous on a nonlinearity interval of the following parametric SOCO problem:
  
\begin{equation}\label{SOCO_counterexample}
\begin{aligned}
\min \qquad &  (\tfrac12 - \tfrac12\epsilon) x_2^1 + (\epsilon - \tfrac 12) x_1^2 -\tfrac 12 \epsilon x_2^2 + (\epsilon-\tfrac12) x_3^2\\[-1\jot]
\st \qquad   &x^1_1+x_1^2 = 4,\\[-1\jot]
&x_3^1 + x_3^2 = 0,\\[-1\jot]
&x^1_1 \ge \sqrt{(x_2^1)^2 + (x_3^1)^2},\\[-1\jot]
&x_1^2 \ge \sqrt{(x_2^2)^2 + (x_3^2)^2},
\end{aligned}
\end{equation}
\noindent
where $\mathcal{E}=\mathbb{R}$. On the interval $(-\frac12,\frac32)$, a strictly complementary optimal solution is given by
\begin{align*}
x^*(\epsilon)=\begin{pmatrix}4\epsilon^3-6\epsilon^2+\epsilon+\frac52\\ 4\epsilon^2-2\epsilon-2\\-4\epsilon^3+6\epsilon^2+\epsilon-\frac32\\ -4\epsilon^3+6\epsilon^2-\epsilon+\frac32\\ 6\epsilon-4\epsilon^2\\ 4\epsilon^3-6\epsilon^2-\epsilon+\frac32\end{pmatrix}, \quad s^*(\epsilon)=\begin{pmatrix} \frac12 \epsilon^2 - \epsilon + \frac58\\ \frac12 - \frac12 \epsilon\\ \frac12\epsilon^2-\epsilon+\frac38\\ \frac12 \epsilon^2 + \frac18\\ -\frac12 \epsilon \\ \frac12 \epsilon^2-\frac18\end{pmatrix},
\end{align*}
which is the unique optimal solution for every $\epsilon \in (-\frac12,\frac32)\setminus\{\frac12\}$ and has the optimal value $\psi(\epsilon)=-2\epsilon^2+4\epsilon - \frac52$. One can verify that $(-\frac12,\frac32)$ is a nonlinearity interval, while the primal optimal set mapping fails to be continuous at $\epsilon=\frac12$. 

\paragraph{\textbf{A numerical procedure}}
If the Jacobian~\eqref{Jacobian_SOCO} is nonsingular at a singleton invariancy set $\{\bar{\epsilon}\}$, then the existence of a nonlinearity interval around $\bar{\epsilon}$ follows from the implicit function theorem~\cite[Theorem~10.2.1]{D60}. We refer the reader to~\cite[Lemma~3.9]{MT20} and its subsequent discussion for an analogous proof. In general, however, the Jacobian of the optimality conditions might be singular on a given subinterval of $\interior(\mathcal{E})$, see e.g., the interval $[1,\infty)$ in the parametric SOCO problem~\eqref{ex:introductory_SOCO}. Even the continuity condition of Theorem~\ref{nonlinearity_continuity_SOCO} may either fail to exist or may be impossible to check at a given point $\bar{\epsilon}$. On the other hand, since the transition points are isolated in $\interior(\mathcal{E})$, see Theorem~\ref{transition_point_finiteness}, and a coordinate of an optimal solution could be doubly exponentially small~\cite[Example~3.2]{MT19}, it may be impractical to obtain the desired nonlinearity interval by simply solving $(\mathrm{P}_{\epsilon})$ and $(\mathrm{D}_{\epsilon})$ at arbitrary points of $\interior(\mathcal{E})$. See also problem~\eqref{SOCO_R_T3} in Section~\ref{transition_point_experiments}, where a boundary point of a nonlinearity interval is irrational.

\vspace{5px}
\noindent
Under strict complementarity condition, we present a numerical procedure to compute a nonlinearity interval in $\interior(\mathcal{E})$. The procedure starts from a singleton invariancy set%
\footnote{For the ease of exposition, we simply rule out the existence of an invariancy interval here. Otherwise, as indicated in Remark~\ref{SDO_SOCO_nonlinearity_alternative}, this procedure can also be applied to invariancy intervals.} $\{\bar{\epsilon}\}$ and a given strictly complementary optimal solution with the goal to find the nonlinearity interval surrounding $\bar{\epsilon}$. The procedure iteratively generates a sequence of subintervals of $\interior(\mathcal{E})$ by solving a pair of nonlinear auxiliary problems. A method of real algebraic geometry is invoked in order to check the existence of a transition point in the given subinterval.

\vspace{5px}
\noindent
The nonlinear auxiliary problems in the above procedure are defined locally w.r.t. a given strictly complementary optimal solution $\big(x^*(\bar{\epsilon});y^*(\bar{\epsilon});s^*(\bar{\epsilon})\big)$. Let us define 
\begin{align*}
\delta_{\mathcal{B}(\bar{\epsilon})}&\!:=\!\frac{\sqrt{2}}{2}\min_{i \in \mathcal{B}(\bar{\epsilon})}\!\big\{x^{*i}_1(\bar{\epsilon})-\|x^{*i}_{2:n_i}(\bar{\epsilon})\|\big\},\\[-1\jot]
\delta_{\mathcal{N}(\bar{\epsilon})}&\!:=\!\frac{\sqrt{2}}{2}\min_{i \in \mathcal{N}(\bar{\epsilon})}\!\big\{s^{*i}_1(\bar{\epsilon})-\|s^{*i}_{2:n_i}(\bar{\epsilon})\|\big\},\\[-1\jot]
\delta_{\mathcal{R}(\bar{\epsilon})}&\!:=\!\min\!\Big \{\min_{i \in \mathcal{R}(\bar{\epsilon})}\! \big\{x^{*i}_1(\bar{\epsilon})\big\},  \min_{i \in \mathcal{R}(\bar{\epsilon})}\! \big\{s^{*i}_1(\bar{\epsilon})\big\} \Big \},\\[-1\jot]
\delta(\bar{\epsilon})&\!:=\!\min\!\big\{\delta_{\mathcal{B}(\bar{\epsilon})},\delta_{\mathcal{N}(\bar{\epsilon})},\delta_{\mathcal{R}(\bar{\epsilon})}\big\},
\end{align*}
and for a given $\delta \ge 0$ let a semi-algebraic set $\mathcal{S}(\delta,\bar{\epsilon})$ be defined by   
\begin{equation}\label{semi_algebraic_strict}
\begin{aligned}
\mathcal{S}(\delta,\bar{\epsilon})\!:=\!
\big\{\epsilon \mid \exists \ (x;y;s) \ \st \ A x = b, \ 
A^T y + s  = c + \epsilon \bar{c},\ x \circ s &= 0,\\ 
 \|(x-x^*(\bar{\epsilon});y-y^*(\bar{\epsilon});s-s^*(\bar{\epsilon}))\|^2 &\le \delta^2\big\},
\end{aligned}
\end{equation}
which is nonempty, and it has a finite number of connected components, see Theorem~\ref{transition_point_finiteness}. For every $\epsilon \in \mathcal{S}(\delta,\bar{\epsilon})$ there exists $\big(\tilde{x}(\epsilon);\tilde{y}(\epsilon);\tilde{s}(\epsilon)\big)$ such that
\begin{align*}
A^T\big(\tilde{y}(\epsilon)-y^*(\epsilon)\big) + \big(\tilde{s}(\epsilon)-s^*(\epsilon)\big) = (\epsilon-\bar{\epsilon})\bar{c},
\end{align*} 
which in turn implies
\begin{align}\label{bound_on_epsilon}
|\epsilon-\bar{\epsilon}|=\frac{\big\|A^T\big(\tilde{y}(\epsilon)-y^*(\epsilon)\big) + \big(\tilde{s}(\epsilon)-s^*(\epsilon)\big)\big\|}{\|\bar{c}\|} \le \frac{\sqrt{2}\max\{\|A\|,1\}\delta}{\|\bar{c}\|},
\end{align} 
where $\|A\|$ is the spectral norm of $A$. Therefore, $\mathcal{S}(\delta,\bar{\epsilon})$ is compact for every fixed $\delta \ge 0$, being the projection of a compact subset of $\mathbb{R}^{m+2\bar{n}+1}$. Furthermore, for any non-increasing sequence $\delta_k \downarrow 0$, we have a nested sequence of nonempty sets $\mathcal{S}(\delta_k,\bar{\epsilon})$ such that $\mathcal{S}(\delta_k,\bar{\epsilon}) \supseteq \mathcal{S}(\delta_{k+1},\bar{\epsilon})$ for all $k$.
\begin{lemma}\label{distance_cond}
Let $\big(x^*(\bar{\epsilon});y^*(\bar{\epsilon});s^*(\bar{\epsilon})\big)$ be a strictly complementary optimal solution. If $0 \le \delta < \delta(\bar{\epsilon})$, then we have $\pi(\epsilon)=\pi(\bar{\epsilon})$ for every $\epsilon \in \mathcal{S}(\delta,\bar{\epsilon})$.
\end{lemma}
\begin{proof}
The inequality constraint in~\eqref{semi_algebraic_strict} ensures that the strict complementarity condition holds at every $\epsilon \in \mathcal{S}(\delta,\bar{\epsilon})$. To see this, for the given $\epsilon$, let $\big(\tilde{x}(\epsilon);\tilde{y}(\epsilon);\tilde{s}(\epsilon)\big)$ be a solution which satisfies the constraints in~\eqref{semi_algebraic_strict}. Then we have 
\begin{align*}
\big|\tilde{x}_1^i(\epsilon) - x^{*i}_1(\bar{\epsilon})\big| + \big \|\tilde{x}_{2:n_i}^i(\epsilon) - x^{*i}_{2:n_i}(\bar{\epsilon}) \big \|  \le \sqrt{2} \|\tilde{x}^i(\epsilon) - x^{*i}(\bar{\epsilon})\|,
\end{align*}
which, together with $\|\tilde{x}^i(\epsilon) - x^{*i}(\bar{\epsilon})\| \le \delta$ for every $i \in \mathcal{B}(\bar{\epsilon})$, gives
\begin{align*}
\|\tilde{x}_{2:n_i}^i(\epsilon)\| - \|x^{*i}_{2:n_i}(\bar{\epsilon})\| -\tilde{x}_1^i(\epsilon) + x^{*i}_1(\bar{\epsilon}) \le\sqrt{2}\|\tilde{x}^i(\epsilon) - x^{*i}(\bar{\epsilon})\| \le \sqrt{2}\delta.
\end{align*}
Consequently, 
\begin{align}
0 < x^{*i}_1(\bar{\epsilon})  - \|x^{*i}_{2:n_i}(\bar{\epsilon})\| - \sqrt{2}\delta  \le \tilde{x}_1^i(\epsilon) - \|\tilde{x}_{2:n_i}^i(\epsilon)\|, \quad\quad i \in \mathcal{B}(\bar{\epsilon}), \label{positive_B}
\end{align}
which, by the complementarity condition, yields $\tilde{s}^i(\epsilon)=0$ for every $i \in \mathcal{B}(\bar{\epsilon})$. Analogously, we can show that $\tilde{s}_1^i(\epsilon) - \|\tilde{s}_{2:n_i}^i(\epsilon)\| > 0$ and thus $\tilde{x}^i(\epsilon)=0$ for every $i \in \mathcal{N}(\bar{\epsilon})$. For any given $i \in \mathcal{R}(\bar{\epsilon})$ we can derive 
\begin{align}
|\tilde{x}_1^i(\epsilon) - x^{*i}_1(\bar{\epsilon})| &\le \|\tilde{x}^i(\epsilon) - x^{*i}(\bar{\epsilon})\| \le \delta  &\Longrightarrow \ &\tilde{x}_1^i(\epsilon) \ge x^{*i}_1(\bar{\epsilon}) - \delta > 0, \label{positive_Rx}\\[-1\jot]
|\tilde{s}_1^i(\epsilon) - s^{*i}_1(\bar{\epsilon})| &\le \|\tilde{s}^i(\epsilon) - s^{*i}(\bar{\epsilon})\| \le \delta   &\Longrightarrow \ &\tilde{s}_1^i(\epsilon) \ge s^{*i}_1(\bar{\epsilon}) - \delta > 0. \label{positive_Rs}
\end{align}
Finally, it follows from $\tilde{x}^i(\epsilon) \circ \tilde{s}^i(\epsilon) = 0$ and~\eqref{Jordan_prod} that
\begin{align*}
0=\tilde{x}^i(\epsilon)^T \tilde{s}^i(\epsilon) =\tilde{x}^i_1(\epsilon) \tilde{s}^i_1(\epsilon) + \big(\tilde{x}^i_{2:n_i}(\epsilon)\big)^T \tilde{s}^i_{2:n_i}(\epsilon) &= \tilde{x}^i_1(\epsilon) \tilde{s}^i_1(\epsilon) - \frac{\tilde{s}^i_1(\epsilon)\|\tilde{x}^i_{2:n_i}(\epsilon)\|^2}{\tilde{x}_1^i(\epsilon)} \\
&= \frac{\tilde{s}_1^i(\epsilon)\big((\tilde{x}_1^i(\epsilon))^2 - \|\tilde{x}^i_{2:n_i}(\epsilon)\|^2\big)}{\tilde{x}_1^i(\epsilon)},
\end{align*}
which, by~\eqref{positive_Rx} and~\eqref{positive_Rs}, yields $\tilde{x}_1^i(\epsilon)-\|\tilde{x}_{2:n_i}^i(\epsilon)\| = 0$, and analogously, $\tilde{s}_1^i(\epsilon)-\|\tilde{s}_{2:n_i}^i(\epsilon)\| = 0$ for every $i \in \mathcal{R}(\bar{\epsilon})$. Since $\mathcal{B}(\bar{\epsilon}) \cup \mathcal{N}(\bar{\epsilon}) \cup \mathcal{R}(\bar{\epsilon})=\{1,\ldots,p\}$, it immediately follows from~\eqref{positive_B},~\eqref{positive_Rx}, and~\eqref{positive_Rs} that $\big(\tilde{x}(\epsilon);\tilde{y}(\epsilon);\tilde{s}(\epsilon)\big)$ is an optimal solution of $(\mathrm{P}_{\epsilon})-(\mathrm{D}_{\epsilon})$ such that $\mathcal{B}(\bar{\epsilon}) \subseteq \mathcal{B}(\epsilon)$, $\mathcal{N}(\bar{\epsilon}) \subseteq \mathcal{N}(\epsilon)$, and $\mathcal{R}(\bar{\epsilon}) \subseteq \mathcal{R}(\epsilon)$. Using the latter inclusions and $\mathcal{B}(\bar{\epsilon}) \cup \mathcal{N}(\bar{\epsilon}) \cup \mathcal{R}(\bar{\epsilon})=\{1,\ldots,p\}$ again, we can conclude that $\pi(\epsilon)=\pi(\bar{\epsilon})$ for every $\epsilon \in \mathcal{S}(\delta,\bar{\epsilon})$. This completes the proof. \qed
\end{proof}

\vspace{5px}
\noindent
Given a strictly complementary optimal solution $\big(x^*(\bar{\epsilon});y^*(\bar{\epsilon});s^*(\bar{\epsilon})\big)$, the idea is to explore the semi-algebraic set $S_{\pi(\bar{\epsilon})}$ by solving the following nonlinear auxiliary problems
\begin{equation}\label{auxiliary_nonlinear_lower}
\begin{aligned}
\alpha(\delta)\big(\beta(\delta)\big)\!:=\!\min(\max) \ &\epsilon \\[-1\jot]
\st \quad  &Ax = b,\\[-1\jot]
&A^T y + s  = c + \epsilon \bar{c},\\[-1\jot]
&x \circ s = 0,\\[-1\jot]
&\|(x-x^*(\bar{\epsilon});y-y^*(\bar{\epsilon});s-s^*(\bar{\epsilon}))\|^2 \le \delta^2,
\end{aligned}
\end{equation}
where both $\alpha(\delta)$ and $\beta(\delta)$ are attained by the compactness of $\mathcal{S}(\delta,\bar{\epsilon})$. The following theorem proves the correctness of our procedure. 
\begin{theorem}\label{semi_algebraic_finiteness}
For all sufficiently small $\delta>0$, $\big[\alpha(\delta),\beta(\delta)\big] \cap S_{\pi(\bar{\epsilon})}$ is either a singleton or a subinterval of $S_{\pi(\bar{\epsilon})}$ containing $\bar{\epsilon}$.
\end{theorem}
\begin{proof} 
Recall from the definition of $S_{\pi(\bar{\epsilon})}$ in Theorem~\ref{transition_point_finiteness} and Lemma~\ref{distance_cond} that 
\begin{align*} 
\mathcal{S}(\delta,\bar{\epsilon}) \subseteq S_{\pi(\bar{\epsilon})} \qquad \forall \ 0 \le \delta < \delta(\bar{\epsilon}),
\end{align*}
which, by~\eqref{auxiliary_nonlinear_lower} and the compactness of $\mathcal{S}(\delta,\bar{\epsilon})$, implies that $\alpha(\delta),\beta(\delta) \in S_{\pi(\bar{\epsilon})}$. Furthermore, the inequality~\eqref{bound_on_epsilon} indicates that the length of $\big[\alpha(\delta),\beta(\delta)\big]$ can be controlled by using a suitable choice of $\delta > 0$. Then the result follows from the finiteness of the number of connected components of $S_{\pi(\bar{\epsilon})}$, see Theorem~\ref{transition_point_finiteness}, and the fact that $\bar{\epsilon} \in [\alpha(\delta),\beta(\delta)]$ for every $\delta \ge 0$: 
\begin{enumerate}
\item If $\alpha(\delta')=\beta(\delta')$ for some $\delta' > 0$, then the case is trivial because then $\mathcal{S}(\delta,\bar{\epsilon})$ must be a singleton for all $0 \le \delta \le \delta'$;
\item Otherwise, suppose that for every $\delta_k > 0$ in a decreasing sequence $\delta_k \downarrow 0$, we have $\alpha(\delta_k) \neq \beta(\delta_k)$ and there exists $\epsilon_k \in \big(\alpha(\delta_k),\beta(\delta_k)\big) \setminus S_{\pi(\bar{\epsilon})}$, which also implies that $\epsilon_k \neq \bar{\epsilon}$. Assume w.l.o.g. that $\beta(\delta_k)$ is constant over $k$. Then the number of connected components of $S_{\pi(\bar{\epsilon})}$ must be at least 2, because $\alpha(\delta_k),\beta(\delta_k) \in S_{\pi(\bar{\epsilon})}$ and $\alpha(\delta_k) < \epsilon_k < \beta(\delta_k)$ while $\epsilon_k \not \in S_{\pi(\bar{\epsilon})}$. Furthermore, since $\epsilon_k \neq \bar{\epsilon}$, by the inequality~\eqref{bound_on_epsilon} we can choose a sufficiently large $k' > k$ such that $\alpha(\delta_{k'}) > \epsilon_k$, and thus $\epsilon_k \not \in \big[\alpha(\delta_{k'}),\beta(\delta_{k'})\big]$. Using the assumption once again, there must exist $\epsilon_{k'} \in \big(\alpha(\delta_{k'}),\beta(\delta_{k'})\big) \setminus S_{\pi(\bar{\epsilon})}$, which in turn yields a lower bound 3 for the number of connected components of $S_{\pi(\bar{\epsilon})}$. Since this process can be continued infinitely many times, it leads to an infinite number of connected components for $S_{\pi(\bar{\epsilon})}$, which, by Theorem~\ref{transition_point_finiteness}, is a contradiction. \qed
\end{enumerate}
\end{proof}
\noindent
Theorem~\ref{nonlinearity_continuity_SOCO} guarantees $\alpha(\delta) < \bar{\epsilon} < \beta(\delta)$ by requiring the continuity of the optimal set mapping at $\bar{\epsilon}$. The following result is then immediate.

\begin{corollary}\label{corollary_singleton}
Assume that the strict complementarity condition holds at $\bar{\epsilon} \in \interior(\mathcal{E})$. If either $\alpha(\delta) = \bar{\epsilon}$, or $\beta(\delta)=\bar{\epsilon}$, or both holds for some $\delta > 0$, then the optimal set mapping fails to be continuous at $\bar{\epsilon}$.
\end{corollary}

\vspace{5px}
\noindent
Consequently, by Theorem~\ref{semi_algebraic_finiteness} and Corollary~\ref{corollary_singleton}, a positive sufficiently small $\delta$ allows us to decide whether $\bar{\epsilon}$ belongs to a nonlinearity interval, or it is a singleton invariancy set at which the optimal set mapping fails to be continuous. Note that the later case would not necessarily lead to the existence of a transition point, as demonstrated by problem~\eqref{SOCO_counterexample}.

\paragraph{Outline of the numerical procedure}
Based on auxiliary problems~\eqref{auxiliary_nonlinear_lower} and the above description, Algorithm~\ref{alg:nonlinearity_computation} presents the outline of our numerical procedure for the computation of a nonlinearity interval. Given the singleton invariancy set $\{\bar{\epsilon}\}$ at which the strict complementarity condition holds, Algorithm~\ref{alg:nonlinearity_computation} tracks forwards and backwards by iteratively solving auxiliary problems~\eqref{auxiliary_nonlinear_lower}. The procedure stops only when $\alpha_k$ and $\alpha_{k-1}$ (or $\beta_k$ and $\beta_{k-1}$) from two consecutive steps are sufficiently close. The connectivity of the subintervals generated by Algorithm~\ref{alg:nonlinearity_computation} is checked by using the quantifier elimination algorithm~\cite[Algorithm~14.5]{BPR06}. We employ the quantifier elimination algorithm to compute the \textit{quantifier free formula}~\cite[Theorem~2.77]{BPR06} whose solution set equals $S_{\pi(\bar{\epsilon})}$. We then invoke~\cite[Algorithm~10.17]{BPR06} to describe an ordered list of the real roots of the univariate polynomials in the quantifier free formula. Afterwards, we decide whether $[\alpha(\delta),\alpha_k)$ or $(\beta_k,\beta(\delta)]$ contains any real root from the list. We omit the description here and refer the reader to~\cite[Chapters~10 and~14]{BPR06} for details.

\vspace{5px}
\noindent
Let $\mathcal{E}_{\mathrm{non}}$ be a bounded nonlinearity interval. Then, starting at an arbitrary $\bar{\epsilon} \in \mathcal{E}_{\mathrm{non}}$, the algorithmic map of Algorithm~\ref{alg:nonlinearity_computation} generates a non-increasing sequence of $\alpha_k$ and a non-decreasing sequence of $\beta_k$ which converge to $\hat{\alpha}$ and $\hat{\beta}$, respectively, in the closure of $\mathcal{E}_{\mathrm{non}}$, as $k \to \infty$.

\begin{algorithm}[] 
\caption{Computation of a nonlinearity interval}
\label{alg:nonlinearity_computation}
\begin{algorithmic}
\State \textbf{Input} A singleton invariancy set $\{\bar{\epsilon}\}$
\vspace{5px}
\State Set $\alpha_1=\bar{\epsilon}$, $\alpha_0=-\infty$, $k=1$.
\vspace{5px}
\State Apply~\cite[Algorithm~14.5]{BPR06} to the quantified formula describing $S_{\pi(\bar{\epsilon})}$. Then apply~\cite[Algorithm~10.17]{BPR06} to the quantifier free formula to describe an ordered list of the real roots of the polynomials.
 \vspace{5px}
\While{$\alpha_k \neq \alpha_{k-1}$} \Comment{Move backwards.}
\State Compute $\big(x^*(\alpha_k);y^*(\alpha_k);s^*(\alpha_k)\big)$ and $\delta(\alpha_k)$. 
\State Set $\delta=2\delta(\alpha_k)$.
\Repeat  \Comment{Invoke the connectivity subroutine.}
\State Set $\delta = \delta/2$.
\State Solve the minimization auxiliary problem in~\eqref{auxiliary_nonlinear_lower} to compute $\alpha(\delta)$.
\Until{$\alpha(\delta)=\alpha_k$ or $[\alpha(\delta),\alpha_k)$ does not contain any real root in the list.}
\State Set $k=k+1$, $\alpha_k=\alpha(\delta)$.
\EndWhile\\
\Return $\hat{\alpha}=\alpha_k$ and its associated optimal solution $\big(x^*(\hat{\alpha});y^*(\hat{\alpha});s^*(\hat{\alpha})\big)$.

\vspace{10px}
\State Set $\beta_1=\bar{\epsilon}$, $\beta_0=\infty$, $k=1$. \\
\While{$\beta_k \neq \beta_{k-1}$} \Comment{Move forwards.}
\State Compute $\big(x^*(\beta_k);y^*(\beta_k);s^*(\beta_k)\big)$ and $\delta(\beta_k)$. 
\State Set $\delta=2\delta(\beta_k)$.
\Repeat  \Comment{Invoke the connectivity subroutine.}
\State Set $\delta = \delta/2$.
\State Solve the maximization auxiliary problem in~\eqref{auxiliary_nonlinear_lower} to compute $\beta(\delta)$.
\Until{$\beta_k=\beta(\delta)$ or $(\beta_k,\beta(\delta)]$ does not contain any real root in the list.}
\State Set $k=k+1$, $\beta_k=\beta(\delta)$.
\EndWhile\\
\Return $\hat{\beta}=\beta_k$ and its associated optimal solution $\big(x^*(\hat{\beta});y^*(\hat{\beta});s^*(\hat{\beta})\big)$.

\vspace{5px}
\If{$\hat{\alpha} < \bar{\epsilon} < \hat{\beta}$} \Comment{$\bar{\epsilon}$ belongs to a nonlinearity interval.}
\State $(\hat{\alpha},\hat{\beta})$ is a subinterval of the nonlinearity interval containing $\bar{\epsilon}$. \\
\Else \Comment{$\bar{\epsilon}$ might be a transition point.}
\State $\bar{\epsilon}$ is a singleton invariancy set at which the continuity of the optimal set mapping fails. 
\EndIf
\end{algorithmic}
\end{algorithm}

\begin{remark}\label{SDO_SOCO_nonlinearity_alternative}
Since Lemma~\ref{distance_cond}, Corollary~\ref{corollary_singleton}, and the auxiliary problems~\eqref{auxiliary_nonlinear_lower} are all applicable to an invariancy interval, Algorithm~\ref{alg:nonlinearity_computation} can be also invoked to find the boundary points of an invariancy interval. Furthermore, it is worth noting that Lemma~\ref{distance_cond}, Corollary~\ref{corollary_singleton}, and the auxiliary problems~\eqref{auxiliary_nonlinear_lower} can be all generalized to a parametric SDO problem. Consequently, Algorithm~\ref{alg:nonlinearity_computation} can be extended to directly compute a nonlinearity interval of a SDO reformulation of $(\mathrm{P}_{\epsilon})-(\mathrm{D}_{\epsilon})$, see Section~\ref{conclusion}. We should note, however, that due to the auxiliary problems and the complexity of the quantifier elimination algorithm, see~\cite[Algorithm~14.5]{BPR06}, the extended Algorithm~\ref{alg:nonlinearity_computation} will have a higher worst-case complexity.
\end{remark}

\subsection{On the identification of a transition point}\label{nonlinearity_interval_strict_fails}
Algorithm~\ref{alg:nonlinearity_computation} relies on the existence of a strictly complementary optimal solution at a given initial point $\bar{\epsilon}$. However, a transition point might coexist with the failure of the strict complementarity condition, see e.g., problem~\eqref{ex:introductory_SOCO}. More specifically, by Theorem~\ref{nonlinearity_continuity_SOCO}, either the strict complementarity or the continuity of the optimal set mapping fails at a transition point. Therefore, Algorithm~\ref{alg:nonlinearity_computation} may be inapplicable to a transition point. To resolve this issue, we present an alternative approach to check the existence of a transition point under both the primal and dual nondegeneracy conditions. We evaluate the higher-order derivatives of the Lagrange multipliers associated with a nonlinear optimization reformulation of $(\mathrm{D}_{\epsilon})$. Obviously, we assume the failure of the strict complementarity condition, since otherwise we would have a nonlinearity interval by the uniqueness of the optimal solution, Lemma~\ref{continuity_sufficient_condition}, and Theorem~\ref{nonlinearity_continuity_SOCO}. 

\vspace{5px}
\noindent
From this point on, we fix $\bar{\epsilon}$ and assume that both the primal and dual nondegeneracy conditions hold at $\bar{\epsilon}$, i.e., there exists a unique optimal solution $\big(x^*(\bar{\epsilon});y^*(\bar{\epsilon});s^*(\bar{\epsilon})\big)$ which is both primal and dual nondegenerate. Further, we define
\begin{align*}
\big\{\bar{\mathcal{B}},\bar{\mathcal{N}},\bar{\mathcal{R}},\{\bar{\mathcal{T}_1},\bar{\mathcal{T}_2},\bar{\mathcal{T}_3}\}\big\}\!:=\pi(\bar{\epsilon}).
\end{align*}

\paragraph{\textbf{Nonlinear reformulation}}
As shown in~\cite{MT19a}, the unique optimal solution of $(\mathrm{D}_{\bar{\epsilon}})$ can be obtained from a globally optimal solution of $(\mathrm{DN}_{\epsilon})$ at $\bar{\epsilon}$:
\begin{align*}
(\mathrm{DN}_{\epsilon}) \quad \min  \qquad -b^Tw\\[-1\jot]
\st \qquad \  A_i^T w &= c^i + \epsilon \bar{c}^i, & i &\in \bar{\mathcal{B}} \cup \bar{\mathcal{T}_1} \cup \bar{\mathcal{T}_2},\\[-1\jot]
A_i^Tw + z^i &= c^i + \epsilon \bar{c}^i, & i &\in \bar{\mathcal{R}} \cup \bar{\mathcal{N}} \cup \bar{\mathcal{T}_3},\\[-1\jot]
(z^i)^T R^i z^i &= 0, & i &\in \bar{\mathcal{R}} \cup \bar{\mathcal{T}_3},\\[-1\jot]
z &\in \bar{\mathcal{W}},
\end{align*}
where $R^i$ is defined in~\eqref{R_definition}, $w \in \mathbb{R}^m$, $z^i\!:=\!(z_1^i; z_{2:n_i}^i) \in \mathbb{R}^{n_i}$ for $i \in \bar{\mathcal{R}} \cup \bar{\mathcal{N}} \cup \bar{\mathcal{T}_3}$, and $\bar{\mathcal{W}}$ is a nonempty open convex cone defined as
\begin{align*}
\mathcal{\bar{W}}\!:=\!\big\{ z \mid z_1^i > 0, \ i \in \bar{\mathcal{R}} \cup \bar{\mathcal{T}_3}, \ \ z^i \in \interior(\mathbb{L}^{n_i}_+), \  i \in \bar{\mathcal{N}}\big \}.
\end{align*}
\noindent
Notice that $(\mathrm{DN}_{\bar{\epsilon}})$ has a unique globally optimal solution because $(\mathrm{D}_{\bar{\epsilon}})$ has a unique optimal solution. Let us define
\begin{align*}
\bar{\mathcal{I}}\!:=\{1,\ldots,p\}.
\end{align*}

\noindent
The Lagrange multipliers associated with the constraints in $(\mathrm{DN}_{\epsilon})$ are denoted by $u^i$ for $i \in \bar{\mathcal{I}}$ and $v \in \mathbb{R}^{|\bar{\mathcal{R}}| + |\bar{\mathcal{T}_3}|}$, respectively. Further, the concatenation of the column vectors $z^i$ for $i \in \bar{\mathcal{R}} \cup \bar{\mathcal{N}} \cup \bar{\mathcal{T}_3}$ and the concatenation of the column vectors $u^i$ for $i \in \bar{\mathcal{I}}$ are denoted by $z$ and $u$, respectively. The first-order optimality conditions for $(\mathrm{DN}_{\epsilon})$ are given by
\begin{equation}\label{optimal_condition_nlo_soco}
\begin{aligned}
- (A^i)_{\bar{\mathcal{I}} } u &= b,\\[-1\jot]
-u^i-2v_i R^i z^i&=0, & i &\in \bar{\mathcal{R}},\\[-1\jot]
-u^i &= 0, & i &\in \bar{\mathcal{N}},\\[-1\jot]
-u^i-2v_i R^i z^i&=0, & i &\in \bar{\mathcal{T}_3},\\[-1\jot]
(A^i)^T w &= c^i + \epsilon \bar{c}^i, & i &\in \bar{\mathcal{B}} \cup \bar{\mathcal{T}_1} \cup \bar{\mathcal{T}_2},\\[-1\jot]
(A^i)^T w + z^i &= c^i + \epsilon \bar{c}^i, & i &\in \bar{\mathcal{R}} \cup \bar{\mathcal{N}} \cup \bar{\mathcal{T}_3},\\[-1\jot]
(z^i)^T R^i z^i &= 0, & i &\in \bar{\mathcal{R}} \cup \bar{\mathcal{T}_3},\\[-1\jot]
z &\in \bar{\mathcal{W}}.
\end{aligned}
\end{equation}
For the unique globally optimal solution $\big(w^*(\bar{\epsilon});z^*(\bar{\epsilon})\big)$ with $z^*(\bar{\epsilon}) \in \bar{\mathcal{W}}$, there exist unique~\cite[Lemma~3.2]{MT19a} Lagrange multipliers $u^*(\bar{\epsilon})$ and $v^*(\bar{\epsilon})$, such that $\big(w^*(\bar{\epsilon});z^*(\bar{\epsilon});u^*(\bar{\epsilon});v^*(\bar{\epsilon})\big)$ satisfies the first-order optimality conditions~\eqref{optimal_condition_nlo_soco}.

\vspace{5px}
\noindent
The first-order optimality conditions~\eqref{optimal_condition_nlo_soco} can be represented by $G\big((w;z;u;v),\epsilon\big)=0, z \in \bar{\mathcal{W}}$, where the mapping $G:\mathbb{R}^{\bar{n}_c} \times \mathbb{R} \to \mathbb{R}^{\bar{n}_c}$ is defined as
\begin{align*}
G\big((w;z;u;v),\epsilon\big)\!:=\!
\begingroup 
\setlength\arraycolsep{.2pt}
\begin{pmatrix}
-(A^i)_{\bar{\mathcal{I}}} u - b & \\
 -u^i-2v_i R^i z^i & i \in \bar{\mathcal{R}}\\
-u^i & i \in \bar{\mathcal{N}}\\
 -u^i-2v_i R^i z^i & i \in \bar{\mathcal{T}_3}\\
(A^i)^T w-c^i -\epsilon \bar{c}^i &\qquad i \in \bar{\mathcal{B}} \cup \bar{\mathcal{T}_1} \cup \bar{\mathcal{T}_2}\\
(A^i)^T w + z^i-c^i-\epsilon \bar{c}^i &\qquad  i \in \bar{\mathcal{R}} \cup \bar{\mathcal{N}} \cup \bar{\mathcal{T}_3}\\
(z^i)^T R^i z^i & i \in \bar{\mathcal{R}} \cup \bar{\mathcal{T}_3}
\end{pmatrix}\endgroup,
\end{align*}
and
\begin{align*}
\bar{n}_c\!:=\!\bar{n} + \sum_{i \in \bar{\mathcal{R}} \cup \bar{\mathcal{N}} \cup \bar{\mathcal{T}_3}} n_i + |\bar{\mathcal{R}}| + |\bar{\mathcal{T}_3}|+ m.
\end{align*}
The following lemma is in order.
\begin{lemma}[Lemmas~3.2,~3.3, and~3.5 in~\cite{MT19a}]\label{Jacobian_nonsingularity}
The Jacobian $\nabla G$ is nonsingular at $\big(\big(w^*(\bar{\epsilon});z^*(\bar{\epsilon});u^*(\bar{\epsilon});v^*(\bar{\epsilon})\big),\bar{\epsilon}\big)$ if both the primal and dual nondegeneracy conditions hold at $\bar{\epsilon}$. 
\end{lemma}

\paragraph{\textbf{Stability of primal-dual nondegeneracy}}
Under both the primal and dual nondegeneracy conditions at $\bar{\epsilon}$, the nonsingularity of the Jacobian and the uniqueness of the optimal solution $\big(x^*(\epsilon);y^*(\epsilon);s^*(\epsilon)\big)$ is not only guaranteed at $\bar{\epsilon}$ but also on a neighborhood of $\bar{\epsilon}$.

\begin{lemma}\label{stability_nondegeneracy}
Both the primal and dual nondegeneracy conditions hold at $\bar{\epsilon}$, if and only if, they hold on a sufficiently small neighborhood of $\bar{\epsilon}$.
\end{lemma}

\begin{proof}
Assume that both the primal and the dual nondegeneracy conditions hold at $\bar{\epsilon}$. Then there exists a unique optimal solution $\big(x^*(\bar{\epsilon});y^*(\bar{\epsilon});s^*(\bar{\epsilon})\big)$ and thus, by Lemma~\ref{continuity_sufficient_condition}, both $\mathcal{P}^*(\epsilon)$ and $\mathcal{D}^*(\epsilon)$ are continuous at $\bar{\epsilon}$. We show that for all $\epsilon$ sufficiently close to $\bar{\epsilon}$, there exists a primal-dual nondegenerate optimal solution.

\vspace{5px}
\noindent
Unlike $\bar{\mathcal{B}}$, $\bar{\mathcal{N}}$, and $\bar{\mathcal{R}}$, the partitions $\bar{\mathcal{T}_1}$, $\bar{\mathcal{T}_2}$, and $\bar{\mathcal{T}_3}$ are not necessarily ``stable'' in a sense of Lemma~\ref{partition_inclusion_lemma}, e.g., it may hold that $\bar{\mathcal{T}_1} \not \subseteq \mathcal{T}_1(\epsilon)$ for any $\epsilon$ sufficiently close to $\bar{\epsilon}$. Given a fixed $\epsilon$, we define the following ``unstable'' subsets of $\bar{\mathcal{T}_1}$, $\bar{\mathcal{T}_2}$, and $\bar{\mathcal{T}_3}$: 
\begin{equation}\label{unstable_partitions}
\begin{aligned}
\bar{\mathcal{T}_1}: \begin{cases}
\bar{\mathcal{T}_1}\mathcal{B}(\epsilon)\!&:=\bar{\mathcal{T}_1} \cap \mathcal{B}(\epsilon),\\[-1\jot]
\bar{\mathcal{T}_1}\mathcal{N}(\epsilon)\!&:=\bar{\mathcal{T}_1} \cap \mathcal{N}(\epsilon),\\[-1\jot]
\bar{\mathcal{T}_1}\mathcal{R}(\epsilon)\!&:=\bar{\mathcal{T}_1} \cap \mathcal{R}(\epsilon),\\[-1\jot]
\bar{\mathcal{T}_1}\mathcal{T}_2(\epsilon)\!&:=\bar{\mathcal{T}_1} \cap \mathcal{T}_2(\epsilon),\\[-1\jot]
\bar{\mathcal{T}_1}\mathcal{T}_3(\epsilon)\!&:=\bar{\mathcal{T}_1} \cap \mathcal{T}_3(\epsilon),\\[-1\jot]
\end{cases}
\end{aligned}
 \qquad
 \begin{aligned}
&\bar{\mathcal{T}_2}: \begin{cases}
\bar{\mathcal{T}_2}\mathcal{B}(\epsilon)\!&:=\bar{\mathcal{T}_2}  \cap \mathcal{B}(\epsilon),\\[-1\jot]
\bar{\mathcal{T}_2}\mathcal{R}(\epsilon)\!&:=\bar{\mathcal{T}_2} \cap \mathcal{R}(\epsilon),\\[-1\jot]
\end{cases} \\
&\bar{\mathcal{T}_3}: \begin{cases}
\bar{\mathcal{T}_3}\mathcal{N}(\epsilon)\!&:=\bar{\mathcal{T}_3} \cap \mathcal{N}(\epsilon),\\[-1\jot]
\bar{\mathcal{T}_3}\mathcal{R}(\epsilon)\!&:=\bar{\mathcal{T}_3} \cap \mathcal{R}(\epsilon).
\end{cases}
\end{aligned}
\end{equation}
For instance, $\bar{\mathcal{T}_1}\mathcal{B}(\epsilon)$ denotes the indices of second-order cones belonging to $\bar{\mathcal{T}_1}$ whose partition changes to $\mathcal{B}(\epsilon)$ after perturbation $\bar{\epsilon}$ to $\epsilon$. When $\epsilon$ is sufficiently close to $\bar{\epsilon}$, it follows from Lemma~\ref{partition_inclusion_lemma} and~\eqref{unstable_partitions} that the optimal partition at $\epsilon$ can be formed as
\begin{equation}\label{new_partition}
\begin{aligned}
\mathcal{B}(\epsilon)&= \bar{\mathcal{B}} \cup \bar{\mathcal{T}_1}\mathcal{B}(\epsilon) \cup \bar{\mathcal{T}_2}\mathcal{B}(\epsilon),\\[-1\jot]
\mathcal{N}(\epsilon)&= \bar{\mathcal{N}} \cup \bar{\mathcal{T}_1}\mathcal{N}(\epsilon) \cup \bar{\mathcal{T}_3}\mathcal{N}(\epsilon),\\[-1\jot]
\mathcal{R}(\epsilon)&= \bar{\mathcal{R}} \cup \bar{\mathcal{T}_1}\mathcal{R}(\epsilon) \cup \bar{\mathcal{T}_2}\mathcal{R}(\epsilon) \cup \bar{\mathcal{T}_3}\mathcal{R}(\epsilon),\\[-1\jot]
\mathcal{T}_1(\epsilon)&= \bar{\mathcal{T}_1} \setminus \big\{\bar{\mathcal{T}_1}\mathcal{B}(\epsilon) \cup \bar{\mathcal{T}_1}\mathcal{N}(\epsilon) \cup \bar{\mathcal{T}_1}\mathcal{R}(\epsilon) \cup \bar{\mathcal{T}_1}\mathcal{T}_2(\epsilon) \cup \bar{\mathcal{T}_1}\mathcal{T}_3(\epsilon)   \big\},\\[-1\jot]
\mathcal{T}_2(\epsilon)&= \bar{\mathcal{T}_2} \cup \bar{\mathcal{T}_1}\mathcal{T}_2(\epsilon) \setminus \big\{\bar{\mathcal{T}_2}\mathcal{B}(\epsilon) \cup \bar{\mathcal{T}_2}\mathcal{R}(\epsilon)\big\},\\[-1\jot]
\mathcal{T}_3(\epsilon)&= \bar{\mathcal{T}_3} \cup \bar{\mathcal{T}_1}\mathcal{T}_3(\epsilon) \setminus \big\{\bar{\mathcal{T}_3}\mathcal{N}(\epsilon) \cup \bar{\mathcal{T}_3}\mathcal{R}(\epsilon)\big\}.
\end{aligned}
\end{equation}
Let $\hat{\epsilon}$ be so close to $\bar{\epsilon}$ that~\eqref{new_partition} holds, and an arbitrary maximally complementary optimal solution $\big(x^*(\hat{\epsilon});y^*(\hat{\epsilon});s^*(\hat{\epsilon})\big)$ is sufficiently close%
\footnote{Since $\mathcal{P}^*(\epsilon) \times \mathcal{D}^*(\epsilon)$ is single-valued and locally bounded at $\bar{\epsilon}$, see~\eqref{optimal_set_locally_bounded}, for every ball $\mathbb{B}_r$ of radius $r$ centered at $\big(x^*(\bar{\epsilon});y^*(\bar{\epsilon});s^*(\bar{\epsilon})\big)$ there exists a neighborhood $V$ of $\bar{\epsilon}$ such that $\mathcal{P}^*(\epsilon) \times \mathcal{D}^*(\epsilon) \subseteq \mathbb{B}_r$ for every $\epsilon \in V$, see~\cite[Proposition~5.12(a)]{Rock09}. Therefore, it is always possible to find such an $\hat{\epsilon}$ with both desired properties.} to the unique optimal solution $\big(x^*(\bar{\epsilon});y^*(\bar{\epsilon});s^*(\bar{\epsilon})\big)$. In what follows, we prove that $\big(x^*(\hat{\epsilon});y^*(\hat{\epsilon});s^*(\hat{\epsilon})\big)$ must be primal-dual nondegenerate.

\paragraph{Primal nondegeneracy}
Recall from the primal nondegeneracy condition, see~\eqref{primal_nondegenerate}, at $\bar{\epsilon}$ that 
\begin{align*}
\begin{pmatrix} \big(A^i\bar{P}^{*i}(\bar{\epsilon})\big)_{\bar{\mathcal{R}} \cup \bar{\mathcal{T}_2}}, \ (A^i)_{\bar{\mathcal{B}}} \end{pmatrix}
\end{align*}
has full row rank. By the continuity of the optimal set mapping at $\bar{\epsilon}$ and the perturbation theory of invariant subspaces~\cite[Theorem~4.11]{St73}%
\footnote{The perturbation theory of invariant subspaces states that the eigenspace associated with the cluster of positive eigenvalues of $L\big(x^{*i}(\bar{\epsilon})\big)$ stays near that of $L\big(x^{*i}(\hat{\epsilon})\big)$.}, we can assume w.l.o.g. that $\bar{P}^{*i}(\hat{\epsilon})$ stays near $\bar{P}^{*i}(\bar{\epsilon})$ for $i \in \bar{\mathcal{R}} \cup \bar{\mathcal{T}_2}$ such that the matrix  
\begin{align}\label{primal_nondegeneracy_perturbed}
\begin{pmatrix} \big(A^i\bar{P}^{*i}(\hat{\epsilon})\big)_{\bar{\mathcal{R}} \cup \bar{\mathcal{T}_2}},  \ (A^i)_{\bar{\mathcal{B}}} \end{pmatrix}
\end{align}
has full row rank, where the columns of $\bar{P}^{*i}(\hat{\epsilon}) \in \mathbb{R}^{n_i \times n_i - 1}$ are normalized eigenvectors corresponding to the positive eigenvalues of $L\big((x^{*i}(\hat{\epsilon})\big)$. Now, by checking the primal nondegeneracy condition for $x^*(\hat{\epsilon})$, we can conclude from~\eqref{new_partition} that 
\begin{align*}
\begin{pmatrix} \big(A^i\bar{P}^{*i}(\hat{\epsilon})\big)_{\mathcal{R}(\hat{\epsilon}) \cup \mathcal{T}_2(\hat{\epsilon})}, \ (A^i)_{\mathcal{B}(\hat{\epsilon})} \end{pmatrix}
\end{align*}
must have full row rank. Otherwise, there would exist a nonzero $\xi$ of appropriate size such that
\begin{align*}
\begin{cases} \big(A^i\bar{P}^{*i}(\hat{\epsilon})\big)^T \xi = 0, \ \ &i \in \mathcal{R}(\hat{\epsilon}) \cup \mathcal{T}_2(\hat{\epsilon}),\\ (A^i)^T \xi = 0, \ \ &i \in \mathcal{B}(\hat{\epsilon}),  \end{cases}
\end{align*}
which, by~\eqref{new_partition}, would imply that
\begin{align*}
\begin{cases} \big(A^i\bar{P}^{*i}(\hat{\epsilon})\big)^T \xi = 0, \ \ &i \in \mathcal{\bar{R}} \cup \mathcal{\bar{T}}_2 \setminus \bar{\mathcal{T}_2}\mathcal{B}(\hat{\epsilon}),\\ (A^i)^T \xi = 0, \ \ &i \in \mathcal{\bar{B}} \cup \bar{\mathcal{T}_2}\mathcal{B}(\hat{\epsilon})  \end{cases}
\end{align*}
has a nonzero solution. However, this contradicts the linear independence of the rows in~\eqref{primal_nondegeneracy_perturbed}. Therefore, $x^*(\hat{\epsilon})$ must be primal nondegenerate.

\paragraph{Dual nondegeneracy}
The proof for the dual nondegeneracy condition is analogous. The dual nondegeneracy condition at $\bar{\epsilon}$ holds if
\begin{align*}
\begin{pmatrix} \big(A^i R^i s^{*i}(\bar{\epsilon})\big)_{\bar{\mathcal{R}} \cup \bar{\mathcal{T}_3}}, \ (A^i)_{\bar{\mathcal{B}} \cup \bar{\mathcal{T}_1} \cup \bar{\mathcal{T}_2}} \end{pmatrix} 
\end{align*}
has linearly independent columns. By the continuity of the dual optimal set mapping, we can assume w.l.o.g. that $\hat{\epsilon}$ is sufficiently close to $\bar{\epsilon}$ such that $s^{*i}_1(\hat{\epsilon}) > 0$ for all $i \in \bar{\mathcal{T}_3}$, $s^{*i}(\hat{\epsilon})$ stays close to $s^{*i}(\bar{\epsilon})$, and the matrix
\begin{align}
\begin{pmatrix} \big(A^i R^i s^{*i}(\hat{\epsilon})\big)_{\bar{\mathcal{R}} \cup \bar{\mathcal{T}_3}}, \ (A^i)_{\bar{\mathcal{B}} \cup \bar{\mathcal{T}_1} \cup \bar{\mathcal{T}_2}}  \end{pmatrix} \label{dual_nondegeneracy_perturbed}
\end{align}
remains full column rank. Then the matrix
\begin{align*}
\begin{pmatrix} \big(A^iR^i s^{*i}(\hat{\epsilon})\big)_{\mathcal{R}(\hat{\epsilon}) \cup \mathcal{T}_3(\hat{\epsilon})}, \ (A^i)_{\mathcal{B}(\hat{\epsilon})  \cup \mathcal{T}_1(\hat{\epsilon})  \cup \mathcal{T}_2(\hat{\epsilon})}  \end{pmatrix} 
\end{align*}
must have linearly independent columns by~\eqref{new_partition}. Otherwise, there would exist nonzero $(\xi;\kappa)$ such that
\begin{align*}
\big(A^iR^i s^{*i}(\hat{\epsilon})\big)_{\mathcal{R}(\hat{\epsilon}) \cup \mathcal{T}_3(\hat{\epsilon})} \xi + \big(A^i\big)_{\mathcal{B}(\hat{\epsilon})  \cup \mathcal{T}_1(\hat{\epsilon})  \cup \mathcal{T}_2(\hat{\epsilon})} \kappa  = 0, 
\end{align*}
which, by~\eqref{new_partition}, would contradict the linear independence of the columns in~\eqref{dual_nondegeneracy_perturbed}. This completes the proof of dual nondegeneracy of $\big(y^*(\hat{\epsilon});s^*(\hat{\epsilon}) \big)$. \qed
\end{proof}
\noindent
As a result of Lemmas~\ref{continuity_sufficient_condition} and~\ref{stability_nondegeneracy}, there exists $\iota > 0 $ such that both the primal and dual optimal set mappings are continuous on $(\bar{\epsilon}-\iota, \bar{\epsilon}+\iota)$, and for the unique optimal solution $\big(x^*(\epsilon);y^*(\epsilon);s^*(\epsilon)\big)$ it holds that $s^{*i}_1(\epsilon) >0$ for every $i \in \bar{\mathcal{R}} \cup \bar{\mathcal{T}_3}$. We define a continuous mapping $\vartheta(\epsilon):(\bar{\epsilon}-\iota, \bar{\epsilon}+\iota) \to \mathbb{R}^{\bar{n}_c}$ by
\begin{align*}
\vartheta(\epsilon)\!:=\!\big(\hat{w}(\epsilon);\hat{z}(\epsilon);\hat{u}(\epsilon);\hat{v}(\epsilon)\big),
\end{align*}
where
\begin{align*}
\hat{w}(\epsilon)&\!:=\!y^*(\epsilon),\\[-1\jot]
\hat{z}^{i}(\epsilon)&\!:=\!s^{*i}(\epsilon), & i &\in \bar{\mathcal{R}} \cup \bar{\mathcal{N}} \cup \bar{\mathcal{T}}_3,\\[-1\jot]
\hat{u}^{i}(\epsilon)&\!:=\!-x^{*i}(\epsilon), & i &\in \bar{\mathcal{I}},\\[-1\jot]
\hat{v}_i(\epsilon)&\!:= \!\frac 12\frac{x^{*i}_1(\epsilon)}{s^{*i}_1(\epsilon)},  & i &\in \bar{\mathcal{R}} \cup \bar{\mathcal{T}_3}.
\end{align*}
We also note from~\eqref{optimal_condition_nlo_soco} that $\vartheta(\bar{\epsilon})$ yields the unique globally optimal solution of $(\mathrm{DN}_{\bar{\epsilon}})$ along with its unique Lagrange multipliers~\cite[Section~3]{MT19a}, i.e., we have
\begin{align*}
\vartheta(\bar{\epsilon})=\big(w^*(\bar{\epsilon});z^*(\bar{\epsilon});u^*(\bar{\epsilon});v^*(\bar{\epsilon})\big).
\end{align*}
Furthermore, if $\pi(\epsilon)$ is constant on a neighborhood of $\bar{\epsilon}$, then we can prove that $\vartheta(\epsilon)$ is a unique real analytic mapping%
\footnote{On an open set $U \subseteq \mathbb{R}$, a mapping $f(x)$ is real analytic if for any given $x_0 \in U$ 
\begin{align*}
f(x)=\sum_{k=0}^{\infty} \frac{\big(f(x_0)\big)^{(k)}}{k!} (x-x_0)^k
\end{align*}
for all $x$ in a neighborhood of $x_0$, where $(.)^{(k)}$ denotes the $k^{\mathrm{th}}$-order derivative. See~\cite[Definition~1.1.5]{KP02} for further properties of an analytic mapping.} such that $G(\vartheta(\epsilon),\epsilon)=0$.
\begin{lemma}\label{analytic_mapping}
Suppose that both the primal and dual nondegeneracy conditions hold at $\bar{\epsilon}$, and the optimal partition is constant on a neighborhood of $\bar{\epsilon}$. Then there exists $0 < \varsigma \le \iota$ such that $\vartheta(\epsilon)$ is a unique real analytic mapping on $(\bar{\epsilon}-\varsigma, \bar{\epsilon}+\varsigma)$ with $G(\vartheta(\epsilon),\epsilon)=0$ for every $\epsilon \in (\bar{\epsilon}-\varsigma, \bar{\epsilon}+\varsigma)$.
\end{lemma}

\begin{proof}
Recall from the discussion after Lemma~\ref{stability_nondegeneracy} that both $\mathcal{P}^*(\epsilon)$ and $\mathcal{D}^*(\epsilon)$ are single-valued and continuous on $(\bar{\epsilon}-\iota, \bar{\epsilon}+\iota)$. Since $\nabla G\big(\vartheta(\bar{\epsilon}),\bar{\epsilon}\big)$ is nonsingular, see Lemma~\ref{Jacobian_nonsingularity}, the analyticity of $\vartheta(\epsilon)$ follows from the analytic implicit function theorem~\cite[Theorem~10.2.4]{D60} (see also~\cite[Theorem~2.3.5]{KP02}) and the invariancy of the optimal partition. More specifically, there exists $\varrho > 0$ and a unique real analytic mapping $\chi(\epsilon)=\big(w(\epsilon);z(\epsilon);u(\epsilon);v(\epsilon)\big)$ on $(\bar{\epsilon} - \varrho ,\bar{\epsilon} + \varrho)$ such that $G(\chi(\epsilon),\epsilon)=0$ for all $\epsilon \in (\bar{\epsilon} - \varrho ,\bar{\epsilon} + \varrho)$ and $\chi(\bar{\epsilon})=\big(w^*(\bar{\epsilon});z^*(\bar{\epsilon});u^*(\bar{\epsilon});v^*(\bar{\epsilon})\big)$. On the other hand, the invariancy of the optimal partition implies $G(\vartheta(\epsilon),\epsilon)=0$ on a small neighborhood of $\bar{\epsilon}$ such that~\cite[Section~3]{MT19a}. As a consequence, by the continuity and the uniqueness of $\chi(\epsilon)$, there exists $0 < \varsigma \le \min\{\varrho,\iota\}$ such that the analytic mapping $\chi(\epsilon)$ and the continuous mapping $\vartheta(\epsilon)$ coincide on $(\bar{\epsilon}-\varsigma,\bar{\epsilon}+\varsigma)$. \qed
\end{proof}
Consequently, the derivatives of $\chi(\epsilon)$ are analytic and well-defined on $(\bar{\epsilon}-\varrho, \bar{\epsilon}+\varrho)$, where $\varrho$ is defined in Lemma~\ref{analytic_mapping}. Furthermore, when $\pi(\epsilon)$ is constant on $(\bar{\epsilon}-\varsigma, \bar{\epsilon}+\varsigma)$, $\vartheta(\epsilon)$ yields a real analytic mapping for the unique globally optimal solution of $(\mathrm{DN}_{\epsilon})$ and its Lagrange multipliers on $(\bar{\epsilon}-\varsigma, \bar{\epsilon}+\varsigma)$.

\paragraph{\textbf{Computation of the higher-order derivatives}}
By Lemma~\ref{partition_inclusion_lemma}, the continuity of $\mathcal{P}^*(\epsilon)$ and $\mathcal{D}^*(\epsilon)$ at $\bar{\epsilon}$ yields the existence of a neighborhood around $\bar{\epsilon}$ on which 
\begin{align}\label{partition_inclusion}
\bar{\mathcal{B}} \subseteq \mathcal{B}(\epsilon),\quad
\bar{\mathcal{N}} \subseteq \mathcal{N}(\epsilon),\quad
\bar{\mathcal{R}} \subseteq \mathcal{R}(\epsilon)
\end{align}
for every $\epsilon$ in the neighborhood. Hence, in order to identify a transition point, we only need to know how the index sets $\mathcal{T}_1(\epsilon)$, $\mathcal{T}_2(\epsilon)$, and $\mathcal{T}_3(\epsilon)$ change near $\bar{\epsilon}$. This can be done by evaluating the higher-order derivatives of the Lagrange multipliers given by $\chi(\epsilon)$ at $\bar{\epsilon}$, as stated in Theorem~\ref{transition_point_identification}.
\begin{theorem}\label{transition_point_identification}
Suppose that $\{\bar{\epsilon}\}$ is a singleton invariancy set, and both the primal and dual nondegeneracy conditions hold at $\bar{\epsilon}$. Then $\bar{\epsilon}$ belongs to a nonlinearity interval, if and only if
\begin{equation}\label{derivatives_condition}
\begin{aligned}
\big(u_j^{i}(\epsilon)\big)^{(k)}|_{\epsilon = \bar{\epsilon}}  &= 0,  & \forall &i \in \bar{\mathcal{T}_1},\forall j=1,\ldots,n_i,\forall k,\\[-1\jot]
\big((u_1^{i}(\epsilon))^2-\|u^{i}_{2:n_i}(\epsilon)\|^2\big)^{(k)}|_{\epsilon = \bar{\epsilon}}  &= 0, & \forall &i \in \bar{\mathcal{T}_2},\forall k,\\[-1\jot]
\big(v_i(\epsilon)\big)^{(k)}|_{\epsilon = \bar{\epsilon}}  &= 0,  & \forall &i \in \bar{\mathcal{T}_3},\forall k,
\end{aligned}
\end{equation}
where $\big(u(\epsilon);v(\epsilon)\big)$ is given by the analytic mapping $\chi(\epsilon)$, and $(.)^{(k)}$ denotes the $k^{\mathrm{th}}$-order derivative w.r.t. $\epsilon$.
\end{theorem}

\begin{proof}
 \begin{enumerate}
\item[$\Rightarrow$] Recall from Lemma~\ref{analytic_mapping} that for the analytic mapping $\chi(\epsilon)$ we have
\begin{equation}\label{u_init}
\begin{aligned}
u^i(\bar{\epsilon}) &= 0,  & \forall &i \in \bar{\mathcal{T}_1},\\[-1\jot]
(u_1^i(\bar{\epsilon}))^2-\|u^i_{2:n_i}(\bar{\epsilon})\|^2  &= 0,  & \forall &i \in \bar{\mathcal{T}_2},\\[-1\jot]
v_i(\bar{\epsilon})  &= 0,  & \forall &i \in \bar{\mathcal{T}_3}.
\end{aligned}
\end{equation}
Assume that $\pi(\epsilon) = \pi(\bar{\epsilon})$ on $(\bar{\epsilon}-\varsigma,\bar{\epsilon}+\varsigma)$, i.e., $\bar{\epsilon}$ is not a transition point, where $\varsigma$ is defined in Lemma~\ref{analytic_mapping}. Then, for every $\epsilon \in (\bar{\epsilon}-\varsigma,\bar{\epsilon}+\varsigma)$ there exists a unique optimal solution $\big(x^*(\epsilon);y^*(\epsilon);s^*(\epsilon)\big)$ such that  
\begin{equation}\label{T_constancy}
\begin{aligned}
x^{*i}(\epsilon)&=0, & \forall i &\in \bar{\mathcal{T}_1},\\[-1\jot]
x^{*i}_1(\epsilon)-\|x^{*i}_{2:n_i}(\epsilon)\|&=0, & \forall i &\in \bar{\mathcal{T}_2},\\[-1\jot]
x^{*i}(\epsilon)&=0, & \forall i &\in \bar{\mathcal{T}_3}.
\end{aligned}
\end{equation}
\noindent
In the sequel, from the equality of $\chi(\epsilon)$ and $\vartheta(\epsilon)$ on $(\bar{\epsilon}-\varsigma,\bar{\epsilon}+\varsigma)$ and~\eqref{T_constancy} we obtain
\begin{equation}\label{u_constancy}
\begin{aligned}
u^i(\epsilon) &= 0,  & \forall &i \in \bar{\mathcal{T}_1},\\[-1\jot]
(u_1^i(\epsilon))^2-\|u^i_{2:n_i}(\epsilon)\|^2  &= 0, & \forall &i \in \bar{\mathcal{T}_2},\\[-1\jot]
v_i(\epsilon)  &= 0, &  \forall &i \in \bar{\mathcal{T}_3}
\end{aligned}
\end{equation}
for every $\epsilon \in (\bar{\epsilon}-\varsigma,\bar{\epsilon}+\varsigma)$, which confirm~\eqref{derivatives_condition}.
\item[$\Leftarrow$] Let all the higher-order derivatives in~\eqref{derivatives_condition} be equal to zero. Then the analyticity of $\chi(\epsilon)$ on $(\bar{\epsilon}-\varrho,\bar{\epsilon}+\varrho)$, where $\varrho$ is defined in Lemma~\ref{analytic_mapping}, and~\eqref{u_init} imply~\eqref{u_constancy} for every $\epsilon \in (\bar{\epsilon}-\varrho,\bar{\epsilon}+\varrho)$, see~\cite[Corollary~1.2.5]{KP02}. Therefore, if $\epsilon \in (\bar{\epsilon}-\varrho,\bar{\epsilon}+\varrho)$ is sufficiently close to $\bar{\epsilon}$, then by~\eqref{optimal_condition_nlo_soco},~\eqref{u_constancy}, and the continuity of $\chi(\epsilon)$ at $\bar{\epsilon}$ there exists%
\footnote{In fact, using~\eqref{optimal_condition_nlo_soco},~\eqref{partition_inclusion}, and~\eqref{u_constancy} we can generate an optimal solution $\big(x^*(\epsilon);y^*(\epsilon);s^*(\epsilon)\big)$ for $(\mathrm{P}_{\epsilon})$ and $(\mathrm{D}_{\epsilon})$, see~\cite[Section~3]{MT19a}, which then proves to be unique for every $\epsilon$ sufficiently close to $\bar{\epsilon}$.} a unique optimal solution $\big(x^*(\epsilon);y^*(\epsilon);s^*(\epsilon)\big)$ such that~\eqref{partition_inclusion} and~\eqref{T_constancy} hold, and thus $\pi(\epsilon) = \pi(\bar{\epsilon})$.  \qed
\end{enumerate}
\end{proof}
Under the primal and dual nondegeneracy conditions, Theorem~\ref{transition_point_identification} provides a complete characterization, in terms of higher-order derivatives, for the identification of a transition point. The higher-order derivatives of $\chi(\epsilon)$, as defined in the proof of Lemma~\ref{analytic_mapping}, can be computed by
\begin{equation}\label{higher_order_derivatives}
\begin{aligned}
\big(\chi(\epsilon)\big)'|_{\epsilon = \bar{\epsilon}}&\!:=\!\nabla G^{-1}\big(\chi(\bar{\epsilon}),\bar{\epsilon}\big)
 \begin{pmatrix} 0 \\ \bar{c}^i & \ \ i &\in \bar{\mathcal{B}} \cup \bar{\mathcal{T}_1} \cup \bar{\mathcal{T}_2}\\ \bar{c}^i & \ \ i &\in \bar{\mathcal{R}} \cup \bar{\mathcal{N}} \cup \bar{\mathcal{T}_3}\\ 0\end{pmatrix}, \\[-1\jot] 
\big(\chi(\epsilon)\big)^{(k)}|_{\epsilon = \bar{\epsilon}}&\!:=\!\nabla G^{-1}\big(\chi(\bar{\epsilon}),\bar{\epsilon}\big) \eta_k(\bar{\epsilon}), \qquad k> 1,
\end{aligned}
\end{equation}
where
\begin{align*}
\eta_k(\bar{\epsilon})\!:=\!\begingroup 
\setlength\arraycolsep{1pt}
\begin{pmatrix} 0\\ 2\sum_{j=1}^{k-1} \displaystyle{k \choose j} \big(v_i(\epsilon)\big)^{(j)}|_{\epsilon = \bar{\epsilon}} R^i \big(z^i(\epsilon)\big)^{(k-j)}|_{\epsilon=\bar{\epsilon}} & i \in \bar{\mathcal{R}} \\ 0 \\ 2\sum_{j=1}^{k-1} \displaystyle{k \choose j}\big(v_i(\epsilon)\big)^{(j)}|_{\epsilon=\bar{\epsilon}} R^i \big(z^i(\epsilon)\big)^{(k-j)}|_{\epsilon = \bar{\epsilon}} & i \in \bar{\mathcal{T}_3} \\ 0 \\ -\sum_{j=1}^{k-1}\displaystyle{k \choose j}\Big(\big(z^i(\epsilon)\big)^{(j)}|_{\epsilon=\bar{\epsilon}}\Big)^TR^i \big(z^i(\epsilon)\big)^{(k-j)}|_{\epsilon=\bar{\epsilon}} & \ \ \ \ \ \ i \in \bar{\mathcal{R}} \cup \bar{\mathcal{T}_3}\end{pmatrix}. \endgroup
\end{align*}

\begin{remark}\label{SDO_SOCO_transition_alternative}
Alternatively, the numerical procedure proposed in~\cite{HMTT19} for parametric SDO can be specialized for $(\mathrm{P}_{\epsilon})-(\mathrm{D}_{\epsilon})$, or it can be directly applied to a SDO reformulation of $(\mathrm{P}_{\epsilon})-(\mathrm{D}_{\epsilon})$, see Section~\ref{conclusion}. Under a nonsingularity condition, which at least requires strict complementarity and nondegeneracy conditions simultaneously, the procedure in~\cite{HMTT19} sequentially solves an ordinary differential equation, arising from the optimality conditions, along a mesh pattern. By contrast, Algorithm~\ref{alg:nonlinearity_computation} and the procedure from Theorem~\ref{transition_point_identification} are more practical in a sense that they rely on weaker assumptions, i.e., strict complementarity or nondegeneracy conditions only, and no need to choose any mesh size.
\end{remark}
\section{Numerical results}\label{numerical_results}
In this section, we numerically evaluate the convergence of the boundaries generated by Algorithm~\ref{alg:nonlinearity_computation}, and the magnitude of the derivatives introduced in Section~\ref{nonlinearity_interval_strict_fails}. For the simplicity of computation, we invoke Algorithm~\ref{alg:nonlinearity_computation} without the connectivity subroutine. We will show that on the given parametric SOCO problems, a subinterval of a nonlinearity interval is properly generated without a need for the quantifier elimination algorithm.

\vspace{5px}
\noindent
We call the SQP algorithm included in the ``fmincon" solver of MATLAB to solve the auxiliary problems in~\eqref{auxiliary_nonlinear_lower}, and we employ the CVX optimization package~\cite{gb08,CVX} to solve the SOCO problems $(\mathrm{P}_{\epsilon})$ and $(\mathrm{D}_{\epsilon})$. The outer loops of Algorithm~\ref{alg:nonlinearity_computation} continue as long as $|\alpha_k - \alpha_{k-1}| > 10^{-7}$ and $|\beta_k - \beta_{k-1}| > 10^{-7}$ hold. Furthermore, in order to accurately compute the higher-order derivatives of the Lagrange multipliers at $\bar{\epsilon}$, we first round the near zero solutions obtained from CVX according to the optimal partition at $\bar{\epsilon}$. We then take a Newton step to solve $G\big((w;z;u;v),\bar{\epsilon}\big)=0$ and thus correct the resulting errors, see~\cite[Section~3.1]{MT19a}. All the codes are run in MATLAB~9.7 environment on a MacBook Pro with Intel Core i5 CPU $@$ 2.3 GHz and 8GB of RAM.
\setlength{\abovedisplayskip}{0pt}
\setlength{\belowdisplayskip}{0pt}
\subsection{Computation of a nonlinearity interval}
We apply Algorithm~\ref{alg:nonlinearity_computation} to the parametric SOCO problems~\eqref{ex:introductory_SOCO} and~\eqref{optim_value_analytic} for the computation of a nonlinearity interval. Additionally, we consider solving the following parametric SOCO problem for which the dual nondegeneracy condition fails on nonlinearity intervals:
\begin{equation}\label{ex:dual_non_fails_SOCO}
\begin{aligned}
\min \ \ -\epsilon x^1_2 - (1-\epsilon) x^1_3\\[-1\jot]
\st \qquad\qquad\qquad x^1_1 &= 1,\\[-1\jot]
x^2_1 + x_1^3 &= 2,\\[-1\jot]
x^1_2 - x^3_2 &= 0,\\[-1\jot]
x^1_3-x^3_3 &= -1,\\[-1\jot]
x_1^1 &\ge \sqrt{(x^1_2)^2 + (x^1_3)^2},\\[-1\jot]
x^2_1 &\ge 0,\\[-1\jot]
x_1^3 &\ge \sqrt{(x^3_2)^2 + (x^3_3)^2},
\end{aligned}
\end{equation}
for which $(-\infty, 0)$ and $(0,\infty)$ are the nonlinearity intervals and $\epsilon = 0$ is a transition point. 

\vspace{5px}
\noindent
For the parametric SOCO problem~\eqref{ex:introductory_SOCO}, a one-time application of the auxiliary problems~\eqref{auxiliary_nonlinear_lower} at $\epsilon = \frac12$ yields $[0.4066,0.5934]$ as a subinterval of the nonlinearity interval containing $\epsilon = \frac12$. By invoking Algorithm~\ref{alg:nonlinearity_computation}, we get the boundary points of the nonlinearity interval, up to our desired precision, in 39 iterations. The numerical results are demonstrated in Tables~\ref{non_interval3_lower} and~\ref{non_interval3_upper}, where "Optim." and "Viol." denote the optimality and feasibility of solutions of~\eqref{auxiliary_nonlinear_lower}, and $\sigma_{\min}(.)$ is the minimum singular value. One can observe from the entries of Tables~\ref{non_interval3_lower} and~\ref{non_interval3_upper} that $\alpha_k$ and $\beta_k$ always remain within $(0,1)$, even without the connectivity subroutine, and converge to 0 and 1 at almost linear rate. Notice that the continuity of $\mathcal{P}^*(\epsilon)$ and $\mathcal{D}^*(\epsilon)$ at $\epsilon = 0$ and $\epsilon = 1$ leads to accurate approximations of the transition points.
\begingroup
\setlength{\tabcolsep}{5pt} 
\renewcommand{\arraystretch}{.85}
\begin{table}[]
\centering
\small
\caption{The convergence of $\alpha_k$ for the parametric SOCO problem~\eqref{ex:introductory_SOCO}.}
\label{non_interval3_lower}
\begin{tabular}{ccccccc}
\hline
$k$ & $\alpha_k$      & Optim.    & Viol.     & $\delta(\alpha_k)$    & $\sigma_{\min}(\nabla F)$      & $|\alpha_k-\hat{\alpha}|$          \\
\hline
0    & 0.5      &          &          & 2.93E-01 & 1.69E-01 &          \\
1    & 0.406585	&1.65E-08  &2.78E-17 &2.76E-01 &1.59E-01	&4.07E-01 \\ 	
2    & 0.315311	&4.44E-16 &1.11E-16 &2.31E-01 &1.35E-01 &3.15E-01 \\
3    & 0.234601	&2.33E-08 &7.65E-15 &1.76E-01 &1.05E-01 &2.35E-01 \\
4    & 0.169844	&4.16E-16 &1.11E-16 &1.27E-01 &7.72E-02 &1.70E-01 \\
5    & 0.121280	&1.11E-16	&1.11E-16	&9.00E-02	&5.53E-02	&1.21E-01 \\
\hline
     &          &          &          &          &          &          \\
 \hline
35   &0.000004	        &4.31E-12	        &4.31E-12	        &2.88E-06	       &1.82E-06	        &4.06E-06\\
36   & 0.000003	&4.90E-12	&4.90E-12	&1.83E-06	&1.16E-06	&2.58E-06 \\
37   & 0.000002	&3.67E-12	&3.67E-12	&1.06E-06	&6.70E-07	&1.50E-06 \\
38   & 0.000001	&1.77E-12	&1.77E-12	&5.64E-07	&3.57E-07	&8.07E-07 \\
39   & 0.000000	&2.50E-11	&2.60E-12	&7.51E-08	&5.06E-08	&1.06E-07 \\
\hline
\end{tabular}
\end{table}
\endgroup

\begingroup
\setlength{\tabcolsep}{5pt} 
\renewcommand{\arraystretch}{.85}
\begin{table}[]
\centering
\small
\caption{The convergence of $\beta_k$ for the parametric SOCO problem~\eqref{ex:introductory_SOCO}.}
\label{non_interval3_upper}
\begin{tabular}{ccccccc}
\hline
$k$ & $\beta_k$      & Optim.    & Viol.     & $\delta(\beta_k)$    & $\sigma_{\min}(\nabla F)$      & $|\beta_k-\hat{\beta}|$         \\
\hline
0    & 0.5      &          &          & 2.93E-01 & 1.69E-01 &          \\
1    & 0.593415	&1.65E-08	&2.22E-16	&2.76E-01	&1.59E-01	&4.07E-01 \\
2    & 0.684689	&2.22E-16	&1.39E-17	&2.31E-01	&1.34E-01	&3.15E-01 \\
3    & 0.765399	&4.43E-15	&5.68E-15	&1.76E-01	&1.04E-01	&2.35E-01 \\
4    & 0.830156	&2.22E-16	&2.78E-17	&1.27E-01	&7.67E-02	&1.70E-01 \\
5    & 0.878720	&2.29E-16	&4.88E-17	&9.00E-02	&5.51E-02	&1.21E-01 \\
\hline
     &          &          &          &          &          &          \\
 \hline
35  &0.999996	        &3.35E-13	        &3.35E-13	        &3.01E-06 	&1.91E-06	&4.26E-06\\
36   & 0.999997	&2.02E-12	&2.02E-12	&2.04E-06	&1.29E-06	&2.87E-06 \\
37   & 0.999998	&2.42E-12	&2.42E-12	&1.29E-06	&8.19E-07	&1.83E-06 \\
38   & 0.999999	&4.39E-12	&4.39E-12	&5.68E-07	&3.59E-07	&8.01E-07 \\
39   & 1.000000	&2.78E-12	&2.78E-12	&6.46E-08	&4.09E-08	&7.79E-08 \\
\hline
\end{tabular}
\end{table}
\endgroup

\vspace{5px}
\noindent
We can observe from Tables~\ref{non_interval_upper} and~\ref{non_interval6_lower} that the bounds given by Algorithm~\ref{alg:nonlinearity_computation} always stay within the corresponding nonlinearity interval, without a need for the connectivity subroutine. However, the convergence of $\beta_k$ and $\alpha_k$ to their limit points becomes slow resulting in a large number of iterations. For instance, the sequence of $\beta_k$ in Table~\ref{non_interval_upper} progresses rapidly on the nonlinearity interval $(-\infty,\frac12)$, where both the strict complementarity and nondegeneracy conditions hold. However, the convergence becomes slower than linear as $\beta_k$ approaches $\frac12$. The slow convergence in the vicinity of $\frac12$ can be partly explained by the discontinuity of the dual optimal set mapping at $\epsilon = \frac12$, at which the primal nondegeneracy condition fails. More precisely, both $\mathcal{P}^*(\epsilon)$ and $\mathcal{D}^*(\epsilon)$ are single-valued and thus continuous on the intervals $(0,\frac12)$ and $(\frac12,1)$, while for any sequence $\epsilon_k \to \frac12$ we have $\liminf_{k \to \infty} \mathcal{D}^*(\epsilon_k) \cap \ri\big(\mathcal{D}^*(\frac12)\big) = \emptyset$. Analogously, the sequence of $\alpha_k$ in Table~\ref{non_interval6_lower} converges slowly to 0, at which the primal nondegeneracy condition fails and thus the dual optimal set mapping fails to be continuous. In this case, compared to Table~\ref{non_interval_upper}, the convergence is slower, which is partially due to the failure of the dual nondegeneracy condition on $(0,\infty)$.

\begingroup
\setlength{\tabcolsep}{5pt} 
\renewcommand{\arraystretch}{.85}
\begin{table}[]
\centering
\small
\caption{The convergence of $\beta_k$ for the parametric SOCO problem~\eqref{optim_value_analytic}.}
\label{non_interval_upper}
\begin{tabular}{lllllll}
\hline
$k$ & $\beta_k$      & Optim.    & Viol.     & $\delta(\beta_k)$    & $\sigma_{\min}(\nabla F)$      &   $|\beta_k-\hat{\beta}|$       \\
\hline
0    & 0.25     &          &          & 1.10E-01 & 5.49E-02 &          \\
1    & 0.281514	&3.75E-16	&1.11E-16	&8.78E-02	&4.38E-02	&2.18E-01 \\
2    & 0.306559	&3.89E-16	&2.22E-16	&7.08E-02	&3.54E-02	&1.93E-01 \\
3    & 0.326779	&1.14E-13	&1.14E-13	&5.80E-02	&2.91E-02	&1.73E-01 \\
4    & 0.343358	&4.44E-16	&3.47E-18	&4.82E-02	&2.42E-02	&1.57E-01 \\
5    & 0.357148	&4.44E-16	&6.73E-25	&4.06E-02	&2.04E-02	&1.43E-01 \\
\hline
     &          &          &          &          &          &          \\
 \hline
196  & 0.492014	&4.40E-13	&4.40E-13	&1.35E-04	&6.76E-05	&7.99E-03 \\
197  & 0.492053	&4.86E-13	&4.86E-13	&1.34E-04	&6.70E-05	&7.95E-03 \\
198  & 0.492092	&5.36E-13	&5.36E-13	&1.33E-04	&6.63E-05	&7.91E-03 \\
199  & 0.492130	&5.91E-13	&5.91E-13	&1.31E-04	&6.57E-05	&7.87E-03 \\
200  & 0.492168	&6.49E-13	&6.49E-13	&1.30E-04	&6.50E-05	&7.83E-03 \\
\hline
\end{tabular}
\end{table}
\endgroup

\begingroup
\setlength{\tabcolsep}{5pt} 
\renewcommand{\arraystretch}{.85}
\begin{table}[]
\centering
\small
\caption{The convergence of $\alpha_k$ for the parametric SOCO problem~\eqref{ex:dual_non_fails_SOCO}.}
\label{non_interval6_lower}
\begin{tabular}{lllllll}
\hline
$k$ & $\alpha_k$      & Optim.    & Viol.     & $\delta(\alpha_k)$    & $\sigma_{\min}(\nabla F)$      & $|\alpha_k-\hat{\alpha}|$          \\
\hline
0    & 0.5      &          &          & 5.20E-02 & 5.12E-15 &          \\
1    & 0.483539	&4.58E-09	&1.11E-16	&4.58E-02	&3.27E-15	&4.84E-01 \\
2    & 0.469025	&4.47E-09	&1.32E-17	&4.51E-02	&4.65E-15	&4.69E-01 \\
3    & 0.454698	&4.48E-09	&1.19E-17	&4.18E-02	&4.99E-15	&4.55E-01 \\
4    & 0.441388	&4.41E-09	&5.87E-17	&3.87E-02	&5.99E-15	&4.41E-01 \\
5    & 0.429023	&4.37E-09	&5.55E-17	&3.44E-02	&7.40E-15	&4.29E-01 \\
\hline
     &          &          &          &          &          &          \\
\hline
196  & 0.090842	&2.94E-04	&4.26E-17	&8.21E-04	&2.33E-13	&9.08E-02 \\
197  &0.090520	&2.75E-04	&1.11E-16	&8.05E-04	&1.22E-12	&9.05E-02 \\
198  & 0.090203	&9.07E-06	&1.11E-16	&7.95E-04	&1.54E-12	&9.02E-02 \\
199  & 0.089891	&7.99E-06	&1.11E-16	&8.04E-04	&2.51E-13	&8.99E-02 \\
200  & 0.089574	&3.94E-01	&2.93E-17	&7.97E-04	&2.16E-13	&8.96E-02\\
\hline
\end{tabular}
\end{table}
\endgroup

\subsection{Identification of a transition point}\label{transition_point_experiments}
In order to illustrate the identification of a transition point, we apply Theorem~\ref{transition_point_identification} to the following parametric SOCO problem:
\begin{equation}\label{SOCO_R_T3}
\begin{aligned}
\min \quad  &x^1_1 + (1-\epsilon) x^1_2\\[-1\jot]
\st \quad& x^1_2 + x_2^2 = 0,\\[-1\jot]
&x^1_3+x^2_1 = -1,\\[-1\jot]
&x^1_1 \ge \sqrt{(x^1_2)^2 + (x^1_3)^2},\\[-1\jot]
&x^2_1 \ge |x^2_2|,
\end{aligned}
\end{equation}
which has a singleton invariancy set at $\epsilon = 0$. The optimal partition at $\epsilon = 0$ is given by $\pi(0)=\big\{\emptyset,\emptyset,\{1\},\{\emptyset,\emptyset,\{2\}\}\big\}$, which indeed implies the failure of the strict complementarity condition at $\epsilon = 0$. However, one can observe that both the primal and dual nondegeneracy conditions hold at $\epsilon = 0$, and thus Theorem~\ref{transition_point_identification} is applicable. By computing the first-order derivative $\big(v(\epsilon)\big)'|_{\epsilon = 0}=-\frac12$ using the formulas in~\eqref{higher_order_derivatives}, we can conclude from Theorem~\ref{transition_point_identification} that $\epsilon = 0$ is a transition point. This transition point is adjacent to a nonlinearity interval $(1-\sqrt{2},0)$ with unique
\begin{align*}
x^{*1}(\epsilon)=\begin{pmatrix} \frac{1}{\sqrt{2-(\epsilon-1)^2}}\\ \frac12 - \frac{1-\epsilon}{2\sqrt{2-(\epsilon-1)^2}}\\ - \frac12 - \frac{1-\epsilon}{2\sqrt{2-(\epsilon-1)^2}} \end{pmatrix}, \qquad s^{*1}(\epsilon)=\begin{pmatrix} 1 \\ \frac{1-\epsilon}{2}-\frac{\sqrt{2-(\epsilon-1)^2}}{2}\\ \frac{1-\epsilon}{2}+\frac{\sqrt{2-(\epsilon-1)^2}}{2} \end{pmatrix}.
\end{align*}
\noindent
Algorithm~\ref{alg:nonlinearity_computation} computes a lower bound $-0.4135$ for this nonlinearity interval, without jumping over the irrational point $1-\sqrt{2}$%
\footnote{Interestingly, the duality gap is zero at the boundary point $\epsilon = 1-\sqrt{2}$, with the finite optimal value $1/\sqrt{2}$. However, the optimal value of~\eqref{SOCO_R_T3} is not attained, and its dual fails to have a strictly feasible solution.}.

\vspace{5px}
\noindent
If we change the objective function of~\eqref{SOCO_R_T3} to $x^1_1 + (1-\epsilon) x^1_2+\epsilon x^1_3$, then we obtain an invariancy interval $(-\infty, 1)$ around $\epsilon = 0$, where
\begin{align*}
\pi(\epsilon)=\big\{\emptyset,\emptyset,\{1\},\{\emptyset,\emptyset,\{2\}\}\big\}, \qquad \forall \epsilon \in (-\infty, 1),
\end{align*}
and both the primal and dual nondegeneracy conditions hold. The higher-order derivatives of $v(\epsilon)$ at $\epsilon = 0$ are given in Table~\ref{tab:higher_order_derivatives}, which stay reliably close to 0 up to the $10^{\mathrm{th}}$-order derivative. The entries of Table~\ref{tab:higher_order_derivatives} also demonstrate the computation error when solving~\eqref{higher_order_derivatives}, which iteratively propagates as the order of derivative increases. 
\begin{table}[h!]
\centering
\small
\caption{The higher-order derivatives of $v(\epsilon)$ at $\epsilon = 0$.}
\label{tab:higher_order_derivatives}
\resizebox{\columnwidth}{!}{\begin{tabular}{ccccccccccc}
\hline
 $k$ & 1      & 2    & 3     & 4    & 5      & 6   & 7 & 8 & 9 & 10   \\
\hline
$\big(v(\epsilon)\big)^{(k)}|_{\epsilon = 0}$   & -5.6E-17   & -1.1E-16   & -3.3E-16   & -1.3E-15 & -6.7E-15 & -4.0E-14  & -2.8E-13 & -2.2E-12 & -2.0E-11 & -2.0E-10  \\
\hline
\end{tabular}}
\end{table}

\setlength{\abovedisplayskip}{0pt}
\setlength{\belowdisplayskip}{0pt}
\section{Concluding remarks and future research}\label{conclusion}
We generalized the optimal partition approach from a parametric LO to a parametric SOCO problem, where the objective function is perturbed along a fixed direction. We provided sufficient conditions for the existence of a nonlinearity interval and proved that the set of transition points is finite. Under strict complementarity condition, we presented a numerical procedure for the computation of a nonlinearity interval. Furthermore, when the primal and dual nondegeneracy conditions hold, we showed how to identify a transition point from the higher-order derivatives of the Lagrange multipliers associated with $(\mathrm{DN}_{\epsilon})$. The numerical experiments demonstrated that Algorithm~\ref{alg:nonlinearity_computation} and Theorem~\ref{transition_point_identification} can be reliably used to check the existence of a nonlinearity interval and a transition point, respectively.

\paragraph{On the connection to parametric SDO}
It is well-known that optimal values of $(\mathrm{P}_{\epsilon})$ and $(\mathrm{D}_{\epsilon})$ can be obtained by solving a SDO problem~\cite{AG03}. The connection can be established, e.g., by noting the equivalence 
\begin{align*}
s^i \in \mathbb{L}^{n_i}_+ \quad \Longleftrightarrow \quad L(s^i) \in \mathbb{S}^{n_i}_+, 
\end{align*}
which embeds $(\mathrm{P}_{\epsilon})-(\mathrm{D}_{\epsilon})$ into a pair of primal-dual SDO problems, as follows:
\begin{align*}
&(\mathrm{P}'_{\epsilon}) \quad \inf_{X\succeq 0} \Big\{\langle L(c + \epsilon \bar{c}), X\rangle \mid  \langle  \mathrm{diag}\big(L(a^1_j),\ldots, L(a^p_j)\big),X \rangle = b_j, \ j=1,\ldots,m\Big\},\\[-1\jot]
&(\mathrm{D}'_{\epsilon}) \quad  \sup_{\substack{y \in \mathbb{R}^m\\ S\succeq 0}} \bigg\{b^T y \mid  \sum_{j=1}^m \mathrm{diag}\big(L(a^1_j),\ldots,L(a^p_j)\big) y_j+S=L(c + \epsilon \bar{c})\bigg\},
\end{align*}
where $a^i_j$ denotes the $j^{\mathrm{th}}$ row of matrix $A^i$ for $i=1,\ldots,p$, $X,S \in \mathbb{S}^{\bar{n}}$, $\langle X,S \rangle:=\trace(XS)$ is defined as the trace of $XS$, and $\succeq$ means positive semidefinite. By~\cite[Lemma~1 and Theorem~1]{SZ05}, both $(\mathrm{P}'_{\epsilon})$ and $(\mathrm{D}'_{\epsilon})$ satisfy the interior point condition at every $\epsilon \in \interior(\mathcal{E})$, and thus they both have nonempty optimal sets on $\interior(\mathcal{E})$. Let $\mathcal{P'}^*(\epsilon)$ and $\mathcal{D'}^*(\epsilon)$ denote the primal and dual optimal sets of $(\mathrm{P}'_{\epsilon})$ and $(\mathrm{D}'_{\epsilon})$, respectively. Since the dual formulation $(\mathrm{D}'_{\epsilon})$ preserves a block diagonal arrow-head structure for $S$, we have $\big(y(\epsilon);s(\epsilon)\big) \in \mathcal{D}^*(\epsilon)$ if and only if $\big(y(\epsilon),L\big(s(\epsilon)\big)\big) \in \mathcal{D'}^{*}(\epsilon)$. Nevertheless, an optimal solution of $(\mathrm{P}'_{\epsilon})$ could be still a fully dense matrix.

\vspace{5px}
\noindent
Unlike the optimal partition of $(\mathrm{P}_{\epsilon})-(\mathrm{D}_{\epsilon})$ at a fixed $\epsilon$, which is specified by the finite collection $\big\{\mathcal{B}(\epsilon),\mathcal{N}(\epsilon),\mathcal{R}(\epsilon),\mathcal{T}(\epsilon)\big\}$, the optimal partition of $(\mathrm{P}'_{\epsilon})$ and $(\mathrm{D}'_{\epsilon})$, see e.g.,~\cite{MT19}, is specified by a 3-tuple $\big\{\mathcal{B'}(\epsilon),\mathcal{N'}(\epsilon),\mathcal{T'}(\epsilon)\big\}$ of mutually orthogonal subspaces of $\mathbb{R}^{\bar{n}}$. Let $\big(X^*(\epsilon),y^*(\epsilon),S^*(\epsilon)\big)$ be a maximally complementary optimal solution%
\footnote{As defined in~\cite[Definition~1.1]{MT20}, $\big(X^*(\epsilon),y^*(\epsilon),S^*(\epsilon)\big)$ is a maximally complementary optimal solution of $(\mathrm{P}'_{\epsilon})$ and $(\mathrm{D}'_{\epsilon})$ if $\rank\!\big(X^*(\epsilon)\big) + \rank\!\big(S^*(\epsilon)\big)$ is maximal on $\mathcal{P'}^*(\epsilon) \times \mathcal{D'}^*(\epsilon)$.} of $(\mathrm{P}'_{\epsilon})$ and $(\mathrm{D}'_{\epsilon})$. Then 
\begin{align*}
\mathcal{B'}(\epsilon)\!:=\col\!\big(X^*(\epsilon)\big), \quad \mathcal{N'}(\epsilon)\!:=\col\!\big(S^*(\epsilon)\big), \quad \mathcal{T'}(\epsilon)\!:=\big(\mathcal{B'}(\epsilon) +\mathcal{N'}(\epsilon)\big)^{\perp},
\end{align*}
where $\perp$ denotes the orthogonal complement of a subspace, and $\col(.)$ represents the column space of a matrix. Since $S^*(\epsilon)$ is always block diagonal and corresponds to an optimal solution $\big(y^*(\epsilon);s^*(\epsilon)\big) \in \mathcal{D}^*(\epsilon)$, one can identify $\mathcal{N}(\epsilon)$ from $S^*(\epsilon)$ by collecting the indices of positive definite blocks. Nevertheless, $X^*(\epsilon)$ might be a fully dense matrix and its eigenstructure may provide no information about $\mathcal{B}(\epsilon)$, $\mathcal{R}(\epsilon)$, and $\mathcal{T}(\epsilon)$.

\paragraph{On the parametric analysis of $(\mathrm{P}'_{\epsilon})$ and $(\mathrm{D}'_{\epsilon})$}
In view of Remarks~\ref{SDO_SOCO_nonlinearity_alternative} and~\ref{SDO_SOCO_transition_alternative}, the parametric analysis of $(\mathrm{P}'_{\epsilon})$ and $(\mathrm{D}'_{\epsilon})$ may be postulated as an alternative to our procedure. However, due to the following reasons, this speculation is not necessarily true:
\begin{itemize}
\item This alternative parametric analysis is at the expense of increasing computational complexity. More specifically, the worst-case iteration complexity of IPMs for an approximate solution of $(\mathrm{P}_{\epsilon})-(\mathrm{D}_{\epsilon})$ is polynomial in terms of $p$, whereas this bound is dependent on $\bar{n} \gg p$, when an IPM is applied to $(\mathrm{P}'_{\epsilon})$ and $(\mathrm{D}'_{\epsilon})$, see also~\cite{AG03,NN94}.

\item Besides computational considerations, the exact identification of the optimal partition is still an issue for $(\mathrm{P}'_{\epsilon})-(\mathrm{D}'_{\epsilon})$. More precisely, computing the boundary points of an invariancy interval for $(\mathrm{P}'_{\epsilon})-(\mathrm{D}'_{\epsilon})$, see~\cite[Lemma~3.4]{MT20}, or a hypothetical extension of our numerical procedure in Section~\ref{nonlinearity_interval_strict_fails} requires the knowledge of the exact optimal partition of $(\mathrm{P}'_{\epsilon})-(\mathrm{D}'_{\epsilon})$. However, only an approximation of the optimal partition is available for $(\mathrm{P}'_{\epsilon})-(\mathrm{D}'_{\epsilon})$, see~\cite{MT19}, while the exact optimal partition of $(\mathrm{P}_{\epsilon})-(\mathrm{D}_{\epsilon})$ can be identified from a trajectory of interior solutions generated by a primal-dual IPM~\cite{TW14}.
\item Even with the exact optimal partition of $(\mathrm{P}'_{\epsilon})-(\mathrm{D}'_{\epsilon})$, the optimal partition of $(\mathrm{P}_{\epsilon})-(\mathrm{D}_{\epsilon})$ cannot be immediately recovered as indicated above.
\end{itemize}

\paragraph{Extension}
While a one-dimensional perturbation setting is of practical interest, our notions of invariancy and nonlinearity intervals can be modified for SOCO problems with other kind of perturbations, e.g., simultaneous perturbation of the objective and right hand side vectors or higher-dimensional perturbations, see e.g.,~\cite{GGT08}. Under these conditions, we need a nontrivial extension of Algorithm~\ref{alg:nonlinearity_computation} for the computation of a nonlinearity region.

\paragraph{Future research}
Algorithm~\ref{alg:nonlinearity_computation} generates sequences of $\alpha_k$ and $\beta_k$, each may converge to a transition point. It is worth investigating the properties of this algorithmic map and the rate at which $\alpha_k$ or $\beta_k$ converges to a transition point. It is also insightful to seek an extension of Algorithm~\ref{alg:nonlinearity_computation} for $(\mathrm{P}_{\epsilon})-(\mathrm{D}_{\epsilon})$, when the strict complementarity condition fails.

\begin{acknowledgements}
We would like to express our gratitude to the anonymous referees for their insightful comments and suggestions. We are indebted to Professor Saugata Basu for bringing the semi-algebraic properties of transition points to our attention. 
\end{acknowledgements}


\bibliographystyle{spmpsci}
\bibliography{mybibfile}

\end{document}